\documentclass[11pt,a4paper,leqno]{amsart}

\usepackage{amsmath,amsthm,amssymb,amsfonts,mathrsfs}
\usepackage{enumitem}
\usepackage{txfonts}
\usepackage{hyperref}

\hypersetup{
  colorlinks=true,
  linkcolor=blue,
  citecolor=blue,
  urlcolor=blue,
  pdftitle={Higher regularity of regular free boundaries for the Kolmogorov obstacle problem},
  pdfauthor={Kaj Nystrom}
}

\calclayout
\allowdisplaybreaks

\newtheorem{theorem}{Theorem}[section]
\newtheorem{proposition}{Proposition}[section]
\newtheorem{lemma}{Lemma}[section]

\theoremstyle{definition}
\newtheorem{definition}{Definition}[section]
\newtheorem{remark}{Remark}[section]

\numberwithin{equation}{section}

\newcommand{\K}{\mathcal K}

\newcommand{\R}{\mathbb R}
\newcommand{\eps}{\varepsilon}
\newcommand{\dK}{d_{\mathcal K}}

\newcommand{\osc}{\operatorname{osc}}

\newcommand{\md}{\operatorname{m.d.}}
\renewcommand{\d}{\,\mathrm d}

\begin{document}

\title[Regular free boundaries for the Kolmogorov obstacle problem]
{Higher regularity of regular free boundaries\\ for the
Kolmogorov obstacle problem}

\author{Kaj Nystr\"om}
\address{Department of Mathematics, Uppsala University, Box 480\\
SE-751 06 Uppsala, Sweden}
\email{kaj.nystrom@math.uu.se}
\date{}

\begin{abstract}
We establish higher regularity of the regular free boundary for
bounded, nonnegative solutions of the Kolmogorov obstacle problem
\[
 \K u=\chi_{\{u>0\}},
 \qquad
 \K=\Delta_x+x\cdot\nabla_y-\partial_t.
\]
Assume that the contact set satisfies a quantitative thickness condition at
a free-boundary point.  We then prove that, in a neighborhood of that point, the
free boundary is a non-characteristic \(C_{\K}^{1,\beta}\) hypersurface for
some \(\beta\in(0,1)\).  Equivalently, its inward non-degenerate normal is
H\"older continuous with respect to the intrinsic Kolmogorov distance, and
the deviation of the graph from its tangent hyperplane is of order
\(O(r^{1+\beta})\) at intrinsic scale \(r\).

The proof combines the half-space blow-up and differentiability theory at
regular points with boundary Harnack estimates in asymptotically
cylindrical intrinsic Lipschitz domains.  A principal difficulty is that
the diffusion derivatives determining the normal do not satisfy the
homogeneous Kolmogorov equation.  We therefore introduce a
\(\K\)-harmonic-replacement argument that converts the commutator sources
into an \(O(r)\) error under intrinsic blow-up and yields a contraction
estimate for the oscillations of derivative quotients.  This gives
H\"older continuity of the non-degenerate normal.  In addition, a boundary
decay estimate for the transport derivative
\(Yu=(x\cdot\nabla_y-\partial_t)u\) provides the required improvement in
the coupled time-transport direction.
Combining these estimates with the available control in the degree-three
variables yields a full intrinsic improvement of flatness.
\end{abstract}

\maketitle
\enlargethispage{6pt}

\section{Introduction and main results}
\label{sec:introduction}

We consider the Kolmogorov operator
\begin{equation}
\label{eq:operator}
 \K=\Delta_x+Y,
 \qquad
 Y=x\cdot\nabla_y-\partial_t,
 \qquad
 (x,y,t)\in\R^m\times\R^m\times\R.
\end{equation}
Diffusion acts only in the \(x\)-variables, but the commutator relations
\begin{equation}
\label{eq:hormander-commutators}
 [\partial_{x_i},Y]=\partial_{y_i},
 \qquad i=1,\ldots,m,
\end{equation}
generate the missing \(y\)-directions and make \(\K\) hypoelliptic.  We refer
to \(Y\) as the Kolmogorov transport field and to \(Yu\) as the transport
derivative of \(u\).  The operator \(\K\) is the prototype of the kinetic
diffusion introduced by Kolmogorov \cite{Kolmogorov1934} and a canonical
example in H\"ormander's theory \cite{Hormander1967}.

The geometry associated with \(\K\) is non-Euclidean.  Its natural
translations are determined by the noncommutative group law
\begin{equation}
\label{eq:group-law-intro}
 (\widetilde x,\widetilde y,\widetilde t)\circ(x,y,t)
 =
 \bigl(\widetilde x+x,\,
       \widetilde y+y-t\widetilde x,\,
       \widetilde t+t\bigr),
\end{equation}
rather than by ordinary vector addition.  The mixed term
\(-t\widetilde x\) records the displacement in \(y\) produced by transport
and is essential for the left invariance of \(\K\).  The group law is
compatible with the anisotropic dilations
\begin{equation}
\label{eq:dilations-intro}
 \delta_r(x,y,t)=(rx,r^3y,r^2t),
\end{equation}
which assign homogeneous degrees one, three, and two to \(x\), \(y\), and
\(t\), respectively.  Regularity is therefore naturally formulated using
the associated quasi-distance and intrinsic function spaces rather than
their Euclidean counterparts.

The relevance of \(\K\) extends beyond its role as a model hypoelliptic
operator.  Up to the choice of time orientation, it is the space-time
generator of integrated Brownian motion, or the random-acceleration process,
in which \(x\) is a diffusing velocity and \(y\) is the corresponding
position.  It is also the principal part of kinetic Fokker-Planck and
Langevin equations arising in kinetic theory and nonequilibrium statistical
mechanics.  Closely related Kolmogorov operators occur as generators of
degenerate stochastic systems in probability, stochastic control, and
filtering.  In mathematical finance, adjoining the running average of an
asset price as an additional state variable leads naturally to Kolmogorov
equations, while optimal stopping for American Asian options gives rise to
the obstacle problem considered below, see, for example,
\cite{DiFrancescoPascucciPolidoro2008,
FrentzNystromPascucciPolidoro2010}.  Across these applications, diffusion
acts only in the velocity variables and is transmitted to the remaining
variables by transport, making the intrinsic boundary and free-boundary
theory of \(\K\) a fundamental analytical issue.

We use the backward Kolmogorov cylinders
\begin{equation}
\label{eq:backward-cylinders}
 Q_r
 =
 B_r^x\times B_{r^3}^y\times(-r^2,0],
 \qquad
 Q_r(P)=P\circ Q_r,
\end{equation}
where \(B_\rho^x\) and \(B_\rho^y\) denote Euclidean balls in the indicated
variables.  Thus \(Q_r(P)\) is obtained from \(Q_r\) by the left translation
in \eqref{eq:group-law-intro}.  Indeed, if \(P=(x_0,y_0,t_0)\), its points
have the form \((x_0+x,y_0+y-tx_0,t_0+t)\), rather than \(P+(x,y,t)\).
Symmetric boxes adapted to boundary
comparison will be denoted by \(Q_{M,r}(P)\) and are introduced in
Section~\ref{sec:geometry}.

In this paper we study bounded, nonnegative distributional solutions
\(u\) of the normalized obstacle problem
\begin{equation}
\label{eq:obstacle}
 \K u=\chi_{\{u>0\}}
 \quad\text{in }Q_1.
\end{equation}
Throughout, we use the continuous representative of \(u\) supplied by
the interior regularity theory.  The regularity results recalled in
Section~\ref{sec:inputs} also provide the classical interpretation of the
equation used in the positivity set.  We write
\begin{equation}
\label{eq:phases}
 \Omega_u=\{u>0\},
 \qquad
 \Lambda_u=\{u=0\},
 \qquad
 \Gamma_u=\partial\Omega_u
\end{equation}
for the positivity set, contact set, and free boundary, respectively.
Here \(\partial\Omega_u\) is taken relative to the open ambient cylinder
\(Q_1^\circ=B_1^x\times B_1^y\times(-1,0)\).  All regular neighborhoods
considered below are compactly contained in \(Q_1^\circ\).

The normalized formulation \eqref{eq:obstacle} arises naturally from an
obstacle problem with a smooth obstacle.  Indeed, let \(v\) solve the
corresponding obstacle problem with obstacle \(\psi\), so that
\[
 v\geq\psi,
 \qquad
 \K v=0
 \quad\text{in }\{v>\psi\}.
\]
After subtracting the obstacle and setting \(w=v-\psi\), the zero-obstacle
formulation becomes
\[
 w\geq0,
 \qquad
 \K w=f\,\chi_{\{w>0\}},
 \qquad
 f=-\K\psi.
\]
Thus \eqref{eq:obstacle} is the constant-source case \(f\equiv1\).  More
generally, if \(P_0\) is a free-boundary point at which \(f(P_0)>0\), then
normalizing \(w\) by \(f(P_0)\) and performing an intrinsic blow-up produce
a right-hand side converging to \(1\).  The results of the present paper
are stated for the exact constant-source equation
\eqref{eq:obstacle}.

The obstacle problem for \eqref{eq:operator} was introduced in connection
with American Asian options and has been studied in
\cite{DiFrancescoPascucciPolidoro2008,
FrentzNystromPascucciPolidoro2010,NystromPascucciPolidoro2010}.
For general background on obstacle-type free boundaries, we refer to
\cite{PetrosyanShahgholianUraltseva2012}.  Frentz, Nystr{\"o}m, Pascucci,
and Polidoro
\cite{FrentzNystromPascucciPolidoro2010} established an optimal growth
theorem at free-boundary points.  More precisely, if
\(P\in\Gamma_u\) and \(Q_{2r}(P)\) is compactly contained in the domain of
the obstacle problem, then
\begin{equation}
\label{eq:optimal-quadratic-growth}
\sup_{Q_r(P)}u\leq Cr^2.
\end{equation}
The exponent two is optimal with respect to the underlying Kolmogorov
dilations \(\delta_r(x,y,t)\), under which solutions are
naturally rescaled according to
\(u_{P,r}(Z)=r^{-2}u(P\circ\delta_r Z)\).  This optimal quadratic growth,
together with the corresponding interior estimates, yields
\begin{equation}
\label{eq:known-solution-regularity}
D_x^2u,\ Yu\in \mathrm L^\infty_{\mathrm{loc}}.
\end{equation}
Theorem~1.1 of \cite{Bowman2025} subsequently supplied the new and decisive
estimate
\begin{equation}
\label{eq:enhanced-y-regularity}
 \nabla_yu\in \mathrm L^\infty_{\mathrm{loc}}.
\end{equation}
Together with the bound for \(Yu\), this also gives
\(\partial_tu=x\cdot\nabla_yu-Yu\in\mathrm L^\infty_{\mathrm{loc}}\).

In \cite{Bowman2025}  the first regularity theorem for \(\Gamma_u\) was established.
Under a quantitative thickness assumption, the author proves that the free boundary is locally a
graph in a non-degenerate \(x\)-direction,
\begin{equation}
\label{eq:regular-free-boundary-graph}
 \Gamma_u=\{x_m=\psi(x',y,t)\},
 \qquad
 \Omega_u=\{x_m>\psi(x',y,t)\}.
\end{equation}
Moreover, \(\psi\in C^{0,1}_{x'}\cap C^{0,1/2}_{y,t}\), the graph has
two-sided intrinsic corkscrews, and it is differentiable with respect to the
Kolmogorov distance
\cite[Theorem~1.2 and Proposition~1.3]{Bowman2025}.  Blow-ups at regular
points are unique \cite[Proposition~6.11]{Bowman2025} and have the form
\begin{equation}
\label{eq:half-space-blowup}
 u_0(x,y,t)=\frac12(x\cdot e)_+^2,
 \qquad e\in\mathbb S^{m-1}.
\end{equation}

\subsection{Main result}

The purpose of this paper is to continue the regular free-boundary program
initiated in \cite{Bowman2025}.  Our main result establishes a power modulus
of continuity for the tangent hyperplane at regular points and thereby
upgrades the qualitative differentiability theorem
\cite[Proposition~1.3]{Bowman2025} to intrinsic
\(C_{\mathcal K}^{1,\beta}\) regularity.  The proof relies on the boundary
comparison theorem for nonnegative solutions recently established in
\cite{NystromAsymptoticallyCylindrical2026}.  The graph geometry in
\eqref{eq:regular-free-boundary-graph} falls within the asymptotically
cylindrical intrinsic Lipschitz geometry treated by that theorem.

Indeed, writing \(y=(y',y_m)\), the \(C^{0,1/2}_y\) estimate for the graph
function in \eqref{eq:regular-free-boundary-graph} gives
\begin{equation}
\label{eq:ym-half-holder-intro}
 \bigl|
 \psi(x',y',y_m,t)
 -
 \psi(x',y',\widetilde y_m,t)
 \bigr|
 \leq C|y_m-\widetilde y_m|^{1/2}.
\end{equation}
Since \(y_m\) has homogeneous degree three under the intrinsic dilations,
variations of intrinsic size \(r\) correspond to
\(|y_m-\widetilde y_m|\lesssim r^3\).  Consequently,
\begin{equation}
\label{eq:cylindrical-defect-intro}
 \sup_{\substack{0<|y_m-\widetilde y_m|\leq c r^3\\
                  (x',y',t)\ \mathrm{fixed}}}
 \frac{\bigl|
 \psi(x',y',y_m,t)
 -
 \psi(x',y',\widetilde y_m,t)
 \bigr|/r}{\bigl(|y_m-\widetilde y_m|/r^3\bigr)^{1/3}}
 \leq C r^{1/2}.
\end{equation}
Thus the scale-invariant \(C^{0,1/3}_{y_m}\) cylindrical deviation tends to zero
at the rate \(O(r^{1/2})\).  The regular free-boundary graph is therefore
asymptotically cylindrical in the sense of
\cite{NystromAsymptoticallyCylindrical2026}, and the constants in the
boundary comparison theorem are uniform at all sufficiently small scales.

We are now ready to state our main theorem.  The thickness function
\(\delta_r(u,P)\) is recalled in
Section~\ref{subsec:bowman-input}.  The constants in the statement depend
on the dimension, a local bound for \(u\), \(r_0\), \(\sigma_0\), and the fixed
regular neighborhood chosen in the proof.  Its localization data include
the intrinsic distance to the fixed boundary and a uniform modulus of
continuous vanishing of \(Yu\) on the regular free boundary.  This modulus
enters through the initial radius at which a fixed smallness threshold is
reached, as explained in Lemma~\ref{lem:uniform-reference-values}.
Here the original domain is the prescribed ambient cylinder \(Q_1\), and
its fixed boundary means \(\partial Q_1\), as distinguished from the unknown
free boundary \(\Gamma_u\).  Thus the theorem is an interior result.
Its estimates are uniform in the boundary point and at all sufficiently
small scales within the fixed regular neighborhood.  Uniform estimates
for a family of solutions require the preceding data, including the
boundary-vanishing modulus, to be controlled uniformly.  The quantity
\(\dK(P,\widetilde P)\) denotes the symmetrized left-invariant Kolmogorov
quasi-distance introduced in \eqref{eq:distance}, and non-characteristic
\(C_{\K}^{1,\beta}\) hypersurfaces are defined in
Definition~\ref{def:C1beta}.

\begin{theorem}
\label{thm:main}
Let \(u\) be a bounded, nonnegative distributional solution of
\eqref{eq:obstacle}, and let \(P_0\in\Gamma_u\cap Q_{1/2}^\circ\).
Assume that, for some \(r_0>0\) with
\(Q_{2r_0}(P_0)\Subset Q_1^\circ\),
\begin{equation}
\label{eq:thickness-main}
 \delta_{r_0}(u,P_0)\geq\sigma_0\geq\sigma(r_0)>0,
\end{equation}
where \(\sigma\) is the universal modulus in the regularity theorem of
\cite{Bowman2025}.
Then, after applying the same orthogonal rotation to the \(x\)- and
\(y\)-variables, there exist \(M\geq1\), \(r_*>0\),
\(\beta\in(0,1)\), and \(C<\infty\) such that the following hold.  For every
\(P\in\Gamma_u\cap Q_{M,r_*}(P_0)\) there is a unique unit vector
\(\nu(P)\in\mathbb S^{m-1}\), with \(\nu(P)\cdot\nu(P_0)\geq1/2\), for
which the blow-up of \(u\) at \(P\) is
\begin{equation}
\label{eq:main-half-space-blowup}
 \frac12(x\cdot\nu(P))_+^2.
\end{equation}
Moreover,
\begin{equation}
\label{eq:normal-holder-main}
 |\nu(P)-\nu(\widetilde P)|
 \leq C(\dK(P,\widetilde P))^\beta
\end{equation}
for \(P,\widetilde P\in\Gamma_u\cap Q_{M,r_*}(P_0)\), and
\begin{equation}
\label{eq:flatness-main}
 \bigl|\nu(P)\cdot(x_{\widetilde P}-x_P)\bigr|
 \leq C(\dK(P,\widetilde P))^{1+\beta}.
\end{equation}
Consequently \(\Gamma_u\cap Q_{M,r_*}(P_0)\) is a non-characteristic
\(C_{\K}^{1,\beta}\) hypersurface.
\end{theorem}

The statement is invariant under the left translations and intrinsic
dilations associated with \(\K\).  In graph coordinates, after applying the
same orthogonal rotation to the \(x\)- and \(y\)-variables, the conclusion
takes the form
\begin{equation}
\label{eq:main-graph-form}
 \Gamma_u\cap Q_{M,r_*}(P_0)
 =
 \{x_m=\psi(x',y,t)\}\cap Q_{M,r_*}(P_0).
\end{equation}
Moreover, the intrinsic first-order differential of \(\psi\) is H\"older
continuous.  In particular, for every fixed \((y,t)\), the function
\(x'\longmapsto\psi(x',y,t)\) is \(C^{1,\beta}\), with estimates uniform
in \((y,t)\) on every smaller
coordinate patch.

The same local conclusion holds at every point of the regular set
\(\Gamma_u^{\mathrm{reg}}\) defined by the existence of a half-space blow-up.
Indeed, \cite[Corollary~5.8 and Proposition~6.11]{Bowman2025} give an open
regular neighborhood and uniqueness of the half-space blow-up.  The
localization argument in Remark~\ref{rem:symmetric-localization} then
supplies the common monotonicity cone, boundary vanishing of \(Yu\), and
uniform quadratic cone non-degeneracy needed for the proof.  The constants
and the admissible radius may depend on this chosen neighborhood.  The
thickness assumption is a convenient quantitative condition ensuring
regularity throughout a neighborhood of the distinguished point.

\subsection{Discussion of the proof}

The mechanism behind the conclusions of \cite{Bowman2025} is important for
the present paper.  Under intrinsic rescaling, the enhanced estimate for
\(\nabla_yu\) from \cite[Theorem~1.1]{Bowman2025} forces the
\(y\)-dependence to disappear in the blow-up limit.  Consequently, every
blow-up is independent of \(y\) and solves the classical parabolic obstacle
problem in \((x,t)\), as proved in
\cite[Proposition~5.1]{Bowman2025}.  The almost-Weiss monotonicity formula
\cite[Proposition~4.1]{Bowman2025} then implies that every blow-up is
parabolically homogeneous of degree two
\cite[Proposition~5.4]{Bowman2025}.  The classification theorem of
Caffarelli, Petrosyan, and Shahgholian
\cite{CaffarelliPetrosyanShahgholian2004}, applied in
\cite[Propositions~5.6-5.7]{Bowman2025}, leaves exactly two possibilities:
a half-space solution of the form \eqref{eq:half-space-blowup} or a quadratic
polynomial solution.  The balanced energy introduced in
\cite[Section~5]{Bowman2025} takes two corresponding values and thereby
distinguishes the branches.  Finally, the quantitative thickness condition
excludes the polynomial branch
not only at the distinguished point but throughout a smaller
free-boundary neighborhood \cite[Lemma~6.6]{Bowman2025}.  We give the details
of this reduction in Section~\ref{subsec:bowman-classification} in order to
identify precisely the information imported from that work.  These results
provide the starting point for the present paper, but they do not by
themselves yield a quantitative modulus of continuity for the tangent
hyperplane.  The remaining task is to upgrade the qualitative uniqueness and
differentiability of the regular free boundary to intrinsic
\(C_{\K}^{1,\beta}\) regularity.

From this perspective, our main theorem is the Kolmogorov analogue of the
completion of the regular free-boundary program for step-two Carnot groups
by Danielli, Garofalo, and Petrosyan
\cite{DanielliGarofaloPetrosyan2007}.  The two settings differ, however, at a
crucial structural point.  In the step-two Carnot-group
setting, the derivatives used to recover the intrinsic normal may be
represented by right-invariant vector fields.  These commute with the
left-invariant sub-Laplacian, and the corresponding derivative functions
satisfy a homogeneous equation and vanish on the regular free boundary.

For the Kolmogorov operator \eqref{eq:operator}, the left-invariant diffusion
fields \(\partial_{x_i}\), which determine the intrinsic normal, do not
commute with \(\K\).  More precisely,
\begin{equation}
\label{eq:commutator-intro}
 [\K,\partial_{x_i}]
 =-\partial_{y_i},
 \qquad
 \K(\partial_{x_i}u)
 =-\partial_{y_i}u
 \quad\text{in }\Omega_u,
 \qquad i=1,\ldots,m.
\end{equation}
Thus the relevant derivative functions satisfy inhomogeneous equations, and
the homogeneous boundary Harnack argument cannot be applied directly.

The key observation is that the inhomogeneity in
\eqref{eq:commutator-intro} is perturbative at a regular free-boundary
point.  For \(Z=(x,y,t)\), set
\begin{equation}
\label{eq:blowup-intro}
 u_{P,r}(Z)=r^{-2}{u(P\circ\delta_r Z)}.
\end{equation}
Left invariance and homogeneity give
\begin{equation}
\label{eq:scaled-derivative-equation-intro}
 \K(\partial_{x_i}u_{P,r})(Z)
 =-r\,(\partial_{y_i}u)(P\circ\delta_r Z).
\end{equation}
The estimate \eqref{eq:enhanced-y-regularity} therefore makes the source
\(O(r)\).  The cone-monotonicity result of
\cite[Proposition~6.5]{Bowman2025}, localized as explained after
Theorem~\ref{thm:bowman-input}, provides an aperture
\(\vartheta_0>0\) such that the relevant directional monotonicity holds for
every diffusion direction in
\[
 \mathcal C_{\vartheta_0}(e_m)
 =
 \left\{
 e\in\mathbb S^{m-1}:
 e\cdot e_m>1-\vartheta_0
 \right\}.
\]
Fix \(\vartheta_1\in(0,\vartheta_0)\), and choose \(a_0\in(0,1)\) so small
that
\[
 e_j^\pm
 =
 {(e_m\pm a_0e_j)}/{\sqrt{1+a_0^2}}
 \in\mathcal C_{\vartheta_1}(e_m),
 \qquad j=1,\ldots,m-1.
\]
Thus these directions lie in a fixed subcone compactly contained in
\(\mathcal C_{\vartheta_0}(e_m)\).  Set
\[
 q_0=(u_{P,r})_{x_m},
 \qquad
 q_j^\pm=(u_{P,r})_{x_m}
 \pm a_0(u_{P,r})_{x_j}.
\]
Cone monotonicity and the scale-uniform domination proved in
Lemma~\ref{lem:uniform-reference-values} give
\[
 q_0>0,
 \qquad
 0\leq q_j^\pm\leq 2q_0,
 \qquad
 u_{P,r}\leq Cq_0
\]
in a bounded portion \(G\) of the rescaled positivity set, obtained by
intersecting \(\{u_{P,r}>0\}\) with an intrinsic box of fixed size.  Moreover, \eqref{eq:enhanced-y-regularity} and
\eqref{eq:scaled-derivative-equation-intro} give
\[
 \|\K q_0\|_{\mathrm L^\infty}
 +
 \|\K q_j^\pm\|_{\mathrm L^\infty}
 \leq Cr.
\]

Let \(q\) be any one of the functions \(q_j^\pm\), and let \(q^h\) be its
Perron \(\K\)-harmonic replacement in \(G\) with boundary data \(q\), as
defined in Subsection~\ref{subsec:Kolmogorov-Dirichlet}.  For \(A>0\)
sufficiently large and \(r\) sufficiently small, the barrier
\[
 B=Aq_0-u_{P,r}
\]
satisfies \(0\leq B\leq Aq_0\) and \(\K B\leq-1/2\) in \(G\).
Perron comparison then gives
\begin{equation}
\label{eq:harmonic-replacement-intro}
 |q-q^h|\leq Crq_0
 \qquad\text{in }G,
\end{equation}
as proved in Lemma~\ref{lem:harmonic-replacement}.  Applying the same
argument to the replacement \(q_0^h\) of \(q_0\) yields
\(q_0^h\simeq q_0\).  Together with \(0\leq q\leq2q_0\), this gives
\[
 \bigl|
 {q}/{q_0}-{q^h}/{q_0^h}
 \bigr|\leq Cr.
\]
Thus \(\K\)-harmonic replacement changes the relevant derivative quotient
by an \(O(r)\) error.

We next apply the homogeneous boundary Harnack inequality proved in \cite{NystromAsymptoticallyCylindrical2026} to suitable
positive oscillation envelopes formed from \(q^h\) and \(q_0^h\).  If
necessary, a small multiple of \(q_0^h\) is added to an envelope to make its
forward and backward reference values quantitatively comparable. This
addition does not change its quotient oscillation.  Transferring the
resulting homogeneous oscillation contraction back to \(q/q_0\), and then
returning to the original scale, yields
\begin{equation}
\label{eq:oscillation-recursion-intro}
 \omega(\theta r)
 \leq
 \vartheta\omega(r)+Cr,
 \qquad
 0<\theta,\vartheta<1,
\end{equation}
where \(\omega(r)\) denotes the oscillation, at scale \(r\), of one of the
quotients \(q_j^\pm/q_0\).  Iterating
\eqref{eq:oscillation-recursion-intro} gives a power modulus of continuity
for these quotients and hence for the non-degenerate normal.

A separate argument is required in the intrinsic time direction.  Normal
H\"older continuity controls the graph in the non-degenerate variables,
whereas the \(C^{0,1/2}_y\) graph estimate already gives an error of order
\(d_{\K}^{3/2}\) in the higher-order variables.  For the flow direction we
use the transport derivative \(Yu\).  In \(\Omega_u\),
\begin{equation}
\label{eq:transport-equation-intro}
 \K(Yu)=2\sum_{i=1}^m\partial_{x_i}\partial_{y_i}u
 =2\operatorname{div}_x(\nabla_yu).
\end{equation}
After rescaling, the vector field on the right-hand side has size \(O(r)\).
A particular-solution estimate combined with boundary H\"older decay gives
a second recursion of the form \eqref{eq:oscillation-recursion-intro}, now
for \(\sup|Yu|\).  This upgrades the qualitative boundary vanishing of
\(Yu\) to power decay.  Integration along the \(Y\)-flow and quadratic
non-degeneracy then give the improvement of flatness in the time direction.

\subsection{Organization of the paper}

The paper is organized as follows.  Section~\ref{sec:geometry} recalls the
Kolmogorov geometry and defines intrinsic \(C^{1,\beta}\) hypersurfaces.
Section~\ref{sec:inputs} explains the reduction and classification of
global blow-up limits and states the input from the boundary comparison
theory.  Section~\ref{sec:derivatives} records the derivative
equations, cone monotonicity, and boundary vanishing.  In
Section~\ref{sec:inhomogeneous} we prove the inhomogeneous quotient
oscillation lemma.  Section~\ref{sec:normal} applies it to obtain H\"older
continuity of the normal.  Section~\ref{sec:transport} proves power decay of
the transport derivative.  Section~\ref{sec:flatness} combines these
estimates to prove Theorem~\ref{thm:main}.  Finally,
Section~\ref{sec:remarks} contains further remarks.

\section{Kolmogorov geometry and intrinsic differentiability}
\label{sec:geometry}

Put \(X=(x,y)\in\R^m\times\R^m\).  We recall the group structure introduced
in \eqref{eq:group-law-intro}-\eqref{eq:dilations-intro}, now in the notation
used throughout the proofs.  The group law is
\begin{equation}
\label{eq:group-law}
 (\widetilde X,\widetilde t)\circ(X,t)
 =
 (\widetilde x+x,
  \widetilde y+y-t\widetilde x,
  \widetilde t+t).
\end{equation}
The identity is the origin, and
\begin{equation}
\label{eq:group-inverse}
 (x,y,t)^{-1}=(-x,-y-tx,-t).
\end{equation}
The intrinsic dilations are
\begin{equation}
\label{eq:dilations}
 \delta_r(x,y,t)=(rx,r^3y,r^2t).
\end{equation}
The vector fields \(\partial_{x_i}\) and \(Y\) are left invariant, and
\(\K\) is homogeneous of degree two.  We use the homogeneous quasi-norm
\[
 \|(X,t)\|_{\K}
 =
 |x|+|y|^{1/3}+|t|^{1/2}.
\]
It induces the left-invariant, generally non-symmetric quasi-distance
\[
 \dK^0(P,Q)
 =
 \|Q^{-1}\circ P\|_{\K}.
\]
Throughout the paper we use its symmetrization
\begin{equation}
\label{eq:distance}
 \dK(P,Q)
 =
 \max\bigl\{
 \dK^0(P,Q),\dK^0(Q,P)
 \bigr\}.
\end{equation}
The two quasi-distances are quantitatively equivalent, so this
symmetrization changes only harmless constants.

For \(z=(z_1,\ldots,z_k)\in\mathbb R^k\), we write
\[
 |z|_\infty=\max_{1\leq i\leq k}|z_i|.
\]
Write \(x=(x',x_m)\) and \(y=(y',y_m)\), with
\(x',y'\in\R^{m-1}\).  We write \(Q_{M,r}\) for the boxes
\begin{align}
 Q_{M,r}=\{(x',x_m,y,t):{}
 &|x'|_\infty<r,\quad |x_m|<c_Mr,\quad |y|_\infty<c_Mr^3,\quad |t|<2r^2\},
\label{eq:boxes}
\end{align}
and \(Q_{M,r}(P)=P\circ Q_{M,r}\).  The geometric constant
\(c_M=c(m,M)\) is fixed by the conditions below.

We next make precise the class of graph domains used below.  Set
\(\zeta=(x',y',y_m,t)\).  For a continuous function \(\psi\), define
\begin{equation}
\label{eq:intrinsic-graph-domain}
 \Omega_\psi=\{(x',x_m,y',y_m,t):x_m>\psi(\zeta)\},
 \qquad
 \Delta_\psi=\{x_m=\psi(\zeta)\}.
\end{equation}
For the intrinsic Lipschitz graphs defined below, the quasi-distance
induced on the graph is equivalent, on fixed boxes, to
\begin{align}
 D_\psi(\zeta,\widetilde\zeta)={}&|x'-\widetilde x'|
 +|y'-\widetilde y'+(t-\widetilde t)\widetilde x'|^{1/3}
 +|t-\widetilde t|^{1/2} +|y_m-\widetilde y_m
 +(t-\widetilde t)\psi(\widetilde\zeta)|^{1/3}.
\label{eq:boundary-quasidistance}
\end{align}
The correction terms in \eqref{eq:boundary-quasidistance} come from the
non-Euclidean group law \eqref{eq:group-law}. In particular, the intrinsic
Lipschitz condition is not the Euclidean Lipschitz condition in
\((x',y,t)\).

\begin{definition}
\label{def:intrinsic-lipschitz-domain}
We call \(\Omega_\psi\) a non-characteristic intrinsic Lipschitz graph
domain, with constant \(M\), if
\begin{equation}
\label{eq:intrinsic-lipschitz-condition}
 |\psi(\zeta)-\psi(\widetilde\zeta)|
 \leq M D_\psi(\zeta,\widetilde\zeta)
\end{equation}
whenever the two graph points belong to the coordinate patch under
consideration.
\end{definition}

\begin{remark}
In Definition~\ref{def:intrinsic-lipschitz-domain}, the adjective
non-characteristic records that the graphing
direction \(\partial_{x_m}\) belongs to the diffusive layer and is uniformly
transverse to the graph.  Thus, after applying the same orthogonal rotation
to the \(x\)- and \(y\)-variables,
every domain considered here has the form \eqref{eq:intrinsic-graph-domain}.
\end{remark}

The constant \(c_M\) in \eqref{eq:boxes} is chosen sufficiently
large, depending only on \(m\) and \(M\), so that the graph remains a fixed
distance below the upper \(x_m\)-face of every sufficiently small box
centered on \(\Delta_\psi\).  Indeed, after translating and dilating the center to the
origin, \eqref{eq:intrinsic-lipschitz-condition} bounds the height of the graph over the
base of the normalized box by
\begin{equation}
\label{eq:box-constant-choice}
 C(m)M\bigl(1+c_M^{1/3}\bigr).
\end{equation}
We choose \(c_M\) so large that this quantity is at most \(c_M/2\).  The
projection of the box onto the variables \((x',y',y_m,t)\) is a rectangle,
and every vertical section of the truncated graph domain is an interval
containing the horizontal level \(x_m=3c_Mr/4\), in the coordinates centered
at \(P_0\).  Vertical segments, followed by a segment in this horizontal
level, therefore join any two points of the truncation.  Consequently,
\begin{equation}
\label{eq:connected-graph-truncation}
 \Omega_\psi\cap Q_{M,r}(P_0)
 \quad\text{is connected}
\end{equation}
whenever \(P_0\in\Delta_\psi\), \(r>0\), and the box lies in the coordinate
neighborhood with its entire projected base in the domain of \(\psi\).
We can also enlarge \(c_M\), by another factor depending only on \(m,M\),
to ensure that the set in \eqref{eq:connected-graph-truncation} contains
the reference points \(A_{r,\Lambda}(P_0)\) and
\(A_{r,\Lambda}^\pm(P_0)\) introduced in
\eqref{eq:translated-reference-points},  see
\cite{NystromAsymptoticallyCylindrical2026}.

For \(P\in\Delta_\psi\), apply the left translation
\(Z\mapsto P^{-1}\circ Z\).  Within the resulting coordinate patch, write
\begin{equation}
\label{eq:translated-graph}
 P^{-1}\circ\Omega_\psi=\Omega_{\psi_P},
 \qquad
 \psi_P(0)=0.
\end{equation}
The rescaled graph function \(\psi_{P,r}\) is defined intrinsically by
\begin{equation}
\label{eq:rescaled-graph}
 \delta_{1/r}\bigl(P^{-1}\circ\Omega_\psi\bigr)
 =
 \Omega_{\psi_{P,r}}.
\end{equation}
Equivalently,
\[
 \psi_{P,r}(x',y,t)
 =r^{-1}\psi_P(rx',r^3y,r^2t).
\]
Let \(\mathcal T_{M,R}\) denote the projection of \(Q_{M,R}\) onto the
\((x',y,t)\)-variables.  The localized cylindrical deviation of \(\psi\) at
\(P\), at scale \(r\) and normalized radius \(R\), is
\begin{equation}
\label{eq:localized-cylindrical-defect}
 \operatorname{cyl}_{R,P,r}(\psi)
 =
 \sup
 \frac{
 \left|
 \psi_{P,r}(x',y',y_m,t)
 -
 \psi_{P,r}(x',y',\widetilde y_m,t)
 \right|}
 {|y_m-\widetilde y_m|^{1/3}},
\end{equation}
where the supremum is taken over pairs in \(\mathcal T_{M,R}\) that belong
to the rescaled graph coordinate patch, have the same \((x',y',t)\), and
satisfy \(y_m\neq\widetilde y_m\).  The deviation vanishes when the graph is
cylindrical, that is, independent of \(y_m\).  Its smallness means that,
after translation and rescaling, the domain is close to a cylindrical
domain in the scale-invariant sense required by the boundary comparison
theorem.

\begin{definition}
\label{def:C1beta}
Let \(\beta\in(0,1)\).  We say that a graph
\(S=\Delta_\psi\) is locally \(C_{\K}^{1,\beta}\) if, on every relatively
compact graph coordinate patch \(S_0\subset S\), there is a constant
\(C<\infty\) such that the following holds.  For every \(P\in S_0\), there
exists \(a_P\in\R^{m-1}\) such that
\begin{equation}
\label{eq:C1beta-remainder}
 \left|\psi_P(x',y,t)-a_P\cdot x'\right|
 \leq
 C\left(|x'|+|y|^{1/3}+|t|^{1/2}\right)^{1+\beta}
\end{equation}
whenever
\[
 (x',\psi_P(x',y,t),y,t)\in P^{-1}\circ S_0.
\]
Moreover,
\begin{equation}
\label{eq:C1beta-gradient}
 |a_P-a_{\widetilde P}|
 \leq
 C(\dK(P,\widetilde P))^\beta,
 \qquad P,\widetilde P\in S_0.
\end{equation}
The vector \(a_P\) is the intrinsic first-order differential of the graph
at \(P\).  Equivalently, after fixing the orientation, the intrinsic unit
normal in the diffusion variables,
\begin{equation}
\label{eq:normal-from-intrinsic-gradient}
 \nu(P)
 =
 \frac{(-a_P,1)}{\sqrt{1+|a_P|^2}}
 \in\mathbb S^{m-1},
\end{equation}
is H\"older continuous with respect to \(\dK\), its \(m\)-th component is
bounded away from zero on every relatively compact graph coordinate patch,
and the flatness estimate \eqref{eq:flatness-main} holds.
\end{definition}

The equivalence in Definition~\ref{def:C1beta} is the intrinsic graph
version of the usual first-order Campanato characterization.  The particular
form used in this paper is established in
Lemma~\ref{lem:geometric-upgrade}.

\begin{remark}
\label{rem:Kolmogorov-differentiability}
When we say that the graph is differentiable in the Kolmogorov sense at
\(P\in S\), we mean that, after translating \(P\) to the origin by the
Kolmogorov group law, there exists \(a_P\in\R^{m-1}\) such that
\begin{equation}
\label{eq:Kolmogorov-differentiability}
 \lim_{\rho\downarrow0}
 \frac{1}{\rho}
 \sup_{\substack{
 |x'|+|y|^{1/3}+|t|^{1/2}\leq\rho\\
 (x',\psi_P(x',y,t),y,t)\in P^{-1}\circ S}}
 \left|\psi_P(x',y,t)-a_P\cdot x'\right|
 =0.
\end{equation}
Thus the intrinsic blow-ups of the graph at \(P\) converge to the tangent
hyperplane \(x_m=a_P\cdot x'\).   Only \(x'\) appears in this first-order polynomial because the variables
\(x'\), \(t\), and \(y\) have homogeneous degrees one, two, and three,
respectively. The translation by the group law is essential because
ordinary Euclidean recentering does not preserve \(\K\) or its intrinsic
geometry.  Condition \eqref{eq:Kolmogorov-differentiability} is
qualitative as it gives neither a uniform rate of convergence nor continuity of \(a_P\) with
respect to \(P\).  Definition~\ref{def:C1beta} strengthens it by imposing
the power remainder \(O(\rho^{1+\beta})\) and a
\(\beta\)-H\"older modulus for the intrinsic first-order differential.
\end{remark}

\section{Preliminaries on the obstacle problem and boundary comparison principles}
\label{sec:inputs}

In this section we collect the two principal inputs used in the proof of the
main theorem.  We first recall the regularity, blow-up classification, and
free-boundary geometry obtained in \cite{Bowman2025} under the quantitative
thickness condition.  We then state the boundary Harnack inequality from
\cite{NystromAsymptoticallyCylindrical2026} in the form needed for the
asymptotically cylindrical graph domains arising here.  Along the way, we
identify the quantitative consequences of these results that will be used
in the \(\K\)-harmonic-replacement and oscillation arguments.

\subsection{The regular free-boundary theorem}
\label{subsec:bowman-input}

We state the results from \cite{Bowman2025} in the variables of
\eqref{eq:operator}.  The variables \((v,x,t)\) used there correspond to our
\((x,-y,t)\).  For a set \(E\subset\R^m\), let \(\md(E)\) denote its minimal
diameter, namely the infimum of the distances between two parallel
hyperplanes such that \(E\) is contained in the slab between them.  At
\(P=(x_0,y_0,t_0)\in\Gamma_u\), the thickness function, after the
corresponding change of variables, is
\begin{equation}
\label{eq:thickness}
 \delta_r(u,P)
 =
 \sup_{y\in B_{r^3}(y_0+r^2x_0)}
 \frac{
 \md\bigl(
 \{x\in B_r(x_0):u(x,y,t_0-r^2)=0\}
 \bigr)}{r}.
\end{equation}
This quantity measures the normalized geometric thickness of the contact set
on a backward time slice.  More precisely, for each admissible \(y\), the
minimal diameter measures the smallest width, over all directions in the
\(x\)-space, of the corresponding \(x\)-section of the contact set.  Thus a
lower bound
\[
 \delta_r(u,P)\geq\sigma>0
\]
implies that, for at least one \(y\) in the indicated intrinsic
neighborhood, the corresponding section has minimal diameter
greater than \(\sigma r/2\). In particular, the contact set does not collapse
to an arbitrarily thin slab in any diffusive direction.

The normalization by \(r\) makes \(\delta_r\) invariant under the
Kolmogorov dilations.  The center \(y_0+r^2x_0\) is the point reached from
\(y_0\) by following the transport associated with \(P\) backward from
\(t_0\) to \(t_0-r^2\), hence its displacement is dictated by the
non-Euclidean group law rather than by an ordinary Euclidean translation.
We use \eqref{eq:thickness} only through
Theorem~\ref{thm:bowman-input} below.

For \(P\in\Gamma_u\) and \(0<r\leq r_*\), define the rescaling of \(u\)
centered at \(P\) by
\begin{equation}
\label{eq:bowman-rescaling}
 u_{P,r}(Z)
 =
 r^{-2}{u(P\circ\delta_r Z)},
 \qquad
 Z=(x,y,t).
\end{equation}
The positivity set of \(u_{P,r}\) is therefore
\[
 \Omega_{u_{P,r}}
 =
 \delta_{1/r}\bigl(P^{-1}\circ\Omega_u\bigr),
\]
and \(u_{P,r}\) satisfies the same normalized obstacle equation in its
rescaled domain.

\begin{theorem}
\label{thm:bowman-input}
Let \(u\) be a bounded, nonnegative solution of \eqref{eq:obstacle},
and suppose that \eqref{eq:thickness-main} holds at
\(P_0\in Q_{1/2}^\circ\).
Then, after reducing the radius and applying the same orthogonal
rotation to the \(x\)- and \(y\)-variables, there exist
\(M\geq1\), \(C\geq1\), \(r_*>0\), \(\kappa\in(0,1)\), and \(c_0>0\)
such that the following statements hold in
\(Q_{M,4r_*}(P_0)\Subset Q_1^\circ\).

\begin{enumerate}[label=\textnormal{(\roman*)}]

\item Every point
\(P\in\Gamma_u\cap Q_{M,4r_*}(P_0)\) is regular.  More precisely, there is a
unique vector \(\nu(P)\in\mathbb S^{m-1}\) such that
\begin{equation}
\label{eq:bowman-unique-blowup}
 u_{P,r}
 \longrightarrow
 \frac12(x\cdot\nu(P))_+^2
 \qquad\text{locally in \(C_x^1\) on \(\{t\leq0\}\) as \(r\downarrow0\)}.
\end{equation}
In particular, the convergence holds for the full family of rescalings and
not merely along a subsequence.

\item The free boundary is a graph
\begin{equation}
\label{eq:bowman-graph}
 \Gamma_u=\{x_m=\psi(x',y,t)\},
 \qquad
 \Omega_u=\{x_m>\psi(x',y,t)\},
\end{equation}
and
\begin{equation}
\label{eq:bowman-graph-regularity}
 |\psi(\zeta)-\psi(\widetilde\zeta)|
 \leq
 M\bigl(
 |x'-\widetilde x'|
 +|y-\widetilde y|^{1/2}
 +|t-\widetilde t|^{1/2}
 \bigr)
\end{equation}
in adapted coordinates.  The graph is differentiable in the Kolmogorov
sense at every point, in the sense explained in
Remark~\ref{rem:Kolmogorov-differentiability}.  Moreover, the diffusion direction \(e_m\) is uniformly non-characteristic,
in the sense that
\begin{equation}
\label{eq:uniform-transversality}
 \nu(P)\cdot e_m\geq c_0
 \qquad
 \text{for }
 P\in\Gamma_u\cap Q_{M,4r_*}(P_0).
\end{equation}

\item The free boundary satisfies a two-sided intrinsic corkscrew condition.
More precisely, for every
\[
 P\in\Gamma_u\cap Q_{M,3r_*}(P_0)
 \quad\text{and}\quad
 0<r\leq r_*,
\]
there exist points \(Z_{P,r}^{+}\in\Omega_u\) and
\(Z_{P,r}^{-}\in\Lambda_u\) such that
\begin{equation}
\label{eq:two-sided-corkscrews}
  Q_{\kappa r}(Z_{P,r}^{+})
  \subset
  \Omega_u\cap Q_{M,r}(P),\qquad
  Q_{\kappa r}(Z_{P,r}^{-})
  \subset
  \Lambda_u\cap Q_{M,r}(P).
\end{equation}
Thus both the positivity set and the contact set contain intrinsic cylinders
whose radii are a fixed proportion of the scale \(r\).  The constant
\(\kappa\) is independent of \(P\) and \(r\).

\item The estimates
\begin{equation}
\label{eq:bowman-solution-bounds}
 \|D_x^2u\|_{\mathrm L^\infty(Q_{M,3r_*}(P_0))}
 +\|\nabla_yu\|_{\mathrm L^\infty(Q_{M,3r_*}(P_0))}
 +\|\partial_tu\|_{\mathrm L^\infty(Q_{M,3r_*}(P_0))}
 \leq C
\end{equation}
hold.  Moreover, \(Yu\), initially defined in \(\Omega_u\), extends
continuously to the regular free boundary with
\begin{equation}
\label{eq:Yu-boundary-vanishing}
 Yu=0
 \qquad\text{on }
 \Gamma_u\cap Q_{M,3r_*}(P_0).
\end{equation}

\item There are \(c>0\), an aperture \(\vartheta_0>0\), and
\(\eta\in(0,1)\) such that the following holds.  For every
\(P\in\Gamma_u\cap Q_{M,3r_*}(P_0)\), every \(0<r\leq r_*\), and every
\(e\in\mathbb S^{m-1}\) satisfying $e\cdot e_m>1-\vartheta_0$, one has
\begin{equation}
\label{eq:cone-monotonicity-sign}
 \partial_eu_{P,r}\geq0
\end{equation}
in a fixed rescaled box containing the reference points introduced in
\eqref{eq:reference-points} below.  In addition,
\begin{equation}
\label{eq:quadratic-cone-lower}
 u_{P,r}(x,0,0)
 \geq
 c(x_m)_+^2
\end{equation}
whenever
\[
 x\in B_1,
 \qquad
 x_m\geq\eta|x|.
\]
Thus, after the single rotation made above, \(u\) is monotone in one fixed
cone of diffusion directions and is quadratically non-degenerate in a
fixed cone about \(e_m\), uniformly throughout the regular neighborhood.
\end{enumerate}
All constants may depend on \(m\),
\(\|u\|_{\mathrm L^\infty(Q_1)}\), the thickness data, and the fixed regular
neighborhood used in the localization below.  They are uniform for \(P\)
in the indicated regular neighborhood and for \(0<r\leq r_*\).
\end{theorem}

\begin{remark}
For an orthogonal matrix \(R\in O(m)\), set $\mathcal R(x,y,t)=(Rx,Ry,t)$, $\widetilde u=u\circ\mathcal R^{-1}$. Orthogonality gives
\[
 \K\widetilde u=(\K u)\circ\mathcal R^{-1},
\]
so the obstacle equation retains its form. The map \(\mathcal R\)
also preserves the group law and commutes with the intrinsic dilations.
We may therefore choose \(R\) so that \(R\nu(P_0)=e_m\).
This single rotation is then fixed throughout the regular neighborhood.
\end{remark}

Theorem \ref{thm:bowman-input} collects the local conclusions of Theorems~1.1-1.2 and
Proposition~1.3 of \cite{Bowman2025}.  The continuous vanishing of \(Yu\),
cone monotonicity, quadratic cone non-degeneracy, and uniqueness of the
half-space blow-up are proved there in Propositions~6.7, 6.5, 7.3, and~6.11,
respectively.

\begin{remark}
\label{rem:symmetric-localization}
In this rather lengthy remark we explain the localization needed to use the conclusions in Theorem \ref{thm:bowman-input} in symmetric
boxes, since the cylinders in \cite{Bowman2025} are backward in time.
First, \(P_0\) is an interior-time point.  The upper semicontinuity of the
limiting Weiss energy and the energy gap imply that the regular set is
relatively open in the free boundary, see
\cite[Corollary~5.8]{Bowman2025}.  We may therefore choose an open
neighborhood of \(P_0\), compactly contained in \(Q_1^\circ\), in which
every free-boundary point is regular. At any such point \(P=(x_P,y_P,t_P)\), the interior bounds for
\(D_x^2u\), \(\nabla_yu\), and \(\partial_tu\), together with the fact that
\(u(P)=|\nabla_xu(P)|=0\), see \eqref{eq:gradient-zero-free-boundary} below, give
\[
 0\leq u_{P,r}(x,y,t)
 \leq C\bigl(|x|^2+|t|+r|y|\bigr)
\]
on every fixed box, including its positive-time part.  The derivative
bounds also give compactness in \(C_x^1\) there.  Any limit is independent
of \(y\), solves the parabolic obstacle problem, and agrees on
\(\{t\leq0\}\) with the unique stationary half-space profile at \(P\). Put
\[
 \mathcal H=\Delta_x-\partial_t.
\]
Forward uniqueness for the parabolic obstacle problem with at most
quadratic growth implies that the limit is the same profile for \(t>0\).
Indeed, for two solutions \(U,V\) with the same initial values,
\(\mathcal H(U-V)\geq0\) on \(\{U>V\}\). Comparison, using a growing
polynomial barrier at spatial infinity, gives \(U\leq V\), and reversing
their roles gives equality.  Thus the half-space approximation at an
interior regular point also holds on fixed symmetric boxes. To obtain one common cone, apply this approximation at \(P_0\), after
rotating its blow-up normal to \(e_m\).  On a fixed bounded box, the
half-space profile \(U=\frac12(x_m)_+^2\) satisfies
\(A\partial_eU-U\geq0\) for every \(e\) in a fixed cone about \(e_m\),
with one sufficiently large \(A\).  Thus, for \(v=u_{P_0,r}\), the
\(C_x^1\) approximation gives \(A\partial_ev-v\geq-\varepsilon_0\)
uniformly in that cone.  Also, \(\partial_ev=0\) on the contact set and
\[
 \K(A\partial_ev)
 =-Ar(\partial_{y_e}u)\circ(P_0\circ\delta_r)
 \leq 1/2
\]
when \(r\) is sufficiently small.  The improvement-of-minimum lemma
\cite[Lemma~6.4]{Bowman2025}, applied in backward cylinders whose top times
lie above a smaller symmetric box, gives \(A\partial_ev-v\geq0\) there.
This supplies a common monotonicity cone on the entire smaller box,
including its positive-time part.  The rotation and cone are fixed once
and for all on the resulting neighborhood. The proof of \cite[Proposition~6.7]{Bowman2025} also localizes on compact
subsets of this open regular neighborhood.  Indeed, the corrected Weiss
energies are continuous in the boundary center and decrease to the same
low-energy value as the scale decreases.  Dini's theorem makes this
convergence uniform on each compact free-boundary portion.  The
interior compactness argument in that proof therefore applies to
sequences approaching any point of this portion.  All backward cylinders
used in the argument remain inside the ambient domain, whether their
centers lie before or after the time coordinate of \(P_0\).  It follows
that \(Yu\) vanishes continuously on the entire free-boundary portion of
the symmetric neighborhood, uniformly in the boundary center. Reduce the neighborhood so that \(|Yu|\leq1/2\) in its positivity set.
On each fixed \((y,t)\)-slice we then have
\(\Delta_xu=1-Yu\geq1/2\).  Elliptic non-degeneracy on spatial balls,
followed by comparison along the common monotonicity cone, gives
\eqref{eq:quadratic-cone-lower}, as in
\cite[Lemma~7.1 and Proposition~7.3]{Bowman2025}.  The balls and comparison
segments can be chosen with fixed margins inside the larger neighborhood,
so both the lower-bound constant and the admissible radius are uniform in
\(P\).  The graph construction and the differentiability argument in
\cite[Proposition~7.4 and Proposition~1.3]{Bowman2025} now apply, and the
resulting two-sided geometry gives the corkscrew assertions. These arguments require an open regular neighborhood and a half-space
blow-up at its center.  They do not require a separate thickness
assumption at every nearby point.  Hence they also apply when the center
is any point of \(\Gamma_u^{\mathrm{reg}}\), by
\cite[Corollary~5.8 and Proposition~6.11]{Bowman2025}.  All constants are
uniform after the smaller regular neighborhood has been fixed.
\end{remark}

We will repeatedly reduce \(r_*\) of Theorem \ref{thm:bowman-input} without changing its notation.
The scale-uniform quantitative domination
\(u_{P,r}\leq C\partial_eu_{P,r}\), for directions in a smaller fixed cone,
is proved separately in Lemma~\ref{lem:uniform-reference-values} below.

\begin{remark} Note that the representation in \eqref{eq:bowman-graph} is not obtained by
applying the implicit function theorem to \(u\), since $u=0$ and $\nabla_xu=0$ on $\Gamma_u$. Instead, one uses cone monotonicity and the positivity and contact
regions constructed in \cite[Propositions~7.3-7.4]{Bowman2025}. After shrinking the patch, these estimates give a product neighborhood
whose lower face lies in the interior of \(\Lambda_u\) and whose upper
face lies in \(\Omega_u\). Let \(I=(\ell_-,\ell_+)\) denote its
\(x_m\)-interval. Monotonicity in the direction \(e_m\) allows us to define
\[
 \psi(x',y,t)
 =\inf\{s\in I:u(x',s,y,t)>0\},
\]
and gives $\Omega_u=\{x_m>\psi(x',y,t)\}$ within this neighborhood. Comparison along the common monotonicity cone gives Lipschitz
dependence on \(x'\). The quadratic lower bound in positivity cones,
together with the bounds for \(\nabla_yu\) and \(\partial_tu\), gives
the one-half H\"older dependence on \((y,t)\) stated in
\eqref{eq:bowman-graph-regularity}. In particular, \(\psi\) is continuous
and
\[
 \Gamma_u=\{x_m=\psi(x',y,t)\}
\]
locally. This construction precedes the boundary Harnack argument,
which improves the regularity of the existing graph.
\end{remark}

\begin{remark}
The common exponent \(1/2\) in \(y\) and \(t\) comes from the
Lipschitz regularity of \(u\) in both variables, combined with
quadratic non-degeneracy in the graph direction.  Fix \(x'\) and set
\[
 a=\psi(x',y,t),
 \qquad
 b=\psi(x',\widetilde y,\widetilde t).
\]
After interchanging the two pairs if necessary, assume \(b\geq a\).
Since \(u(x',b,\widetilde y,\widetilde t)=0\), we obtain
\[
 \begin{aligned}
 c(b-a)^2
 &\leq u(x',b,y,t)=u(x',b,y,t)-u(x',b,\widetilde y,\widetilde t)\leq C\bigl(|y-\widetilde y|+|t-\widetilde t|\bigr).
 \end{aligned}
\]
Taking square roots yields the stated exponents.  Since \(y\) has
homogeneous degree three, its \(1/2\)-H\"older estimate is stronger
than the \(1/3\)-H\"older control associated with intrinsic
Lipschitz regularity. More precisely, \eqref{eq:bowman-graph-regularity} implies
\eqref{eq:intrinsic-lipschitz-condition} on every bounded coordinate patch.
With \(\widetilde x=(\widetilde x',\psi(\widetilde\zeta))\), we have
\[
 \begin{aligned}
 |y-\widetilde y|^{1/2}
 &\leq |y-\widetilde y+(t-\widetilde t)\widetilde x|^{1/2}
       +|\widetilde x|^{1/2}|t-\widetilde t|^{1/2}\leq C D_\psi(\zeta,\widetilde\zeta).
 \end{aligned}
\]
The last inequality uses boundedness of the patch and
\(s^{1/2}\leq C s^{1/3}\) for bounded \(s\geq0\).
Increasing \(M\) if necessary, we choose it to control both
\eqref{eq:bowman-graph-regularity} and
\eqref{eq:intrinsic-lipschitz-condition}.  All subsequent geometric
constants, including \(c_M\) and \(\Lambda\), use this choice, and \(r_*\)
is reduced so that the corresponding boxes remain in the coordinate patch.
\end{remark}

\begin{remark} For the remainder of the paper, \(u\) denotes a fixed solution of the
obstacle problem satisfying the hypotheses of
Theorem~\ref{thm:bowman-input}, and we work in a fixed regular
neighborhood of \(P_0\).  Unless otherwise stated, constants may depend
on the geometric and analytic data of this neighborhood, but not on the
boundary point or the scale.  These data include a uniform modulus
\(\omega\) of continuous boundary vanishing of \(Yu\), as used in
Lemma~\ref{lem:uniform-reference-values}, and the initial radii chosen
using that modulus.
\end{remark}
\subsection{Blow-up reduction and classification of global limits}
\label{subsec:bowman-classification}

Here we describe in some detail the blow-up argument underlying
part~\textnormal{(i)} of Theorem~\ref{thm:bowman-input}.  This explains how
the Kolmogorov obstacle problem reduces to the classical parabolic problem
at small scales and how the thickness condition selects the half-space
branch of the global classification.

We record the blow-up reduction and classification here, relying on
\cite{Bowman2025} for the cited results. With $\mathcal H=\Delta_x-\partial_t$,  \(\K=\mathcal H+x\cdot\nabla_y\), and \(\mathcal H\) is the heat
operator obtained in the blow-up limit once
the dependence on the transport variables has disappeared.
\begin{proposition}
\label{prop:bowman-classification}
Let \(P\in\Gamma_u\cap Q_1^\circ\) and let \(r_j\downarrow0\).  After
passing to a subsequence, \(u_{P,r_j}\) converges locally on \(\{t\leq0\}\), together
with its diffusion gradient, to a global function \(u_0=u_0(x,t)\) such that
\begin{equation}
\label{eq:parabolic-global-obstacle}
 \mathcal H u_0=\chi_{\{u_0>0\}}
 \quad\text{in }\R^m\times(-\infty,0],
 \qquad
 0\leq u_0(x,t)\leq C(1+|x|^2+|t|).
\end{equation}
The function \(u_0\) is parabolically homogeneous of degree two,
\begin{equation}
\label{eq:parabolic-homogeneity}
 u_0(rx,r^2t)=r^2u_0(x,t),
 \qquad r>0.
\end{equation}
Consequently, exactly one of the following alternatives occurs:
\begin{align}
 u_0(x,t)&=\frac12(x\cdot e)_+^2,
 &&e\in\mathbb S^{m-1},
 \label{eq:regular-global-branch}\\
 u_0(x,t)&=a t+\frac12x\cdot A x,
 &&A=A^{\mathsf T}\geq0,\quad a\leq0,
 \quad \operatorname{tr}A-a=1.
 \label{eq:singular-global-branch}
\end{align}
The first is the regular, or low-energy, branch.  The second is the
polynomial, or high-energy, branch.  If the thickness hypothesis
\eqref{eq:thickness-main} holds at \(P_0\),
then only \eqref{eq:regular-global-branch} occurs at every free-boundary
point in a fixed neighborhood of \(P_0\).  At those points the vector \(e\)
is unique.
\end{proposition}

\begin{proof}
We divide the argument into four steps.

\smallskip
\noindent
\emph{Step 1: disappearance of the transport variables.}
By the interior estimates in \cite[Theorem~1.1]{Bowman2025}
and the scaling \eqref{eq:bowman-rescaling},
\begin{equation}
\label{eq:y-gradient-disappears}
 \nabla_yu_{P,r}(Z)
 =r(\nabla_yu)(P\circ\delta_rZ),
 \qquad
 \|\nabla_yu_{P,r}\|_{\mathrm L^\infty(K)}\leq C_Kr.
\end{equation}
The local estimates, quadratic growth, and Arzel\`a-Ascoli therefore give a
subsequence converging in \(C_x^1\) on compact subsets to a global limit
which is independent of \(y\).  In the positivity set the rescaled equation
can be written as
\begin{equation}
\label{eq:almost-parabolic-rescaling}
 \mathcal H u_{P,r}
 =1-x\cdot\nabla_yu_{P,r}.
\end{equation}
The last term tends to zero locally uniformly by
\eqref{eq:y-gradient-disappears}.  Stability of the obstacle equation,
non-degeneracy, and quadratic growth now give
\eqref{eq:parabolic-global-obstacle}, including the fact that the origin is
a free-boundary point of \(u_0\). This is
\cite[Proposition~5.1]{Bowman2025} in the present variables.

\smallskip
\noindent
\emph{Step 2: almost-Weiss monotonicity and homogeneity.}
The argument in \cite{Bowman2025} localizes \(u\) by a cutoff and introduces
a scale-invariant Weiss functional.  By intrinsic translation, we may
write the functional in coordinates centered at \(P\), where \(P\)
corresponds to the origin.  In these coordinates it has the form
\begin{equation}
\label{eq:bowman-weiss-functional}
 \mathcal W(r,v)
 =\frac1{r^4}
 \int_{-4r^2}^{-r^2}\int_{\R^{2m}}
 \left(|\nabla_xv|^2+2v+\frac{v^2}{t}\right)
 \mathscr G(x,y,t)\,\d x\,\d y\,\d t,
\end{equation}
where, in our variables, the weight from \cite[Section~4]{Bowman2025} is
\[
 \mathscr G(x,y,t)
 =\frac{\pi^{-m/2}}{(-t)^{2m}}
   \exp\left(\frac{|x|^2}{4t}+\frac{|y|^2}{t^3}\right),
 \qquad t<0.
\]
This Gaussian weight differs from the fundamental solution
\(\Gamma_{\K}\) used in Section~\ref{sec:transport}.
The relevant features of \(\mathscr G\) are its Kolmogorov homogeneity,
its Gaussian decay, and the identity which converts differentiation in
\(r\) into the square of the dilation generator.  In the original
variables, we use the centered notation
\[
 \mathcal W(r,v;P)=\mathcal W(r,v(P\circ\cdot)).
\]
We choose \(\chi\) by left translation of a fixed smooth cutoff to \(P\),
with support in the solution domain and equal to one on a smaller
intrinsic box centered at \(P\).  After absorbing the mixed transport
term into the nonnegative square in the Weiss identity, the estimate
\eqref{eq:y-gradient-disappears}, together with the cutoff estimates, gives
\[
 \frac{\mathrm d}{\mathrm d r}\mathcal W(r,u\chi;P)\geq-Cr
\]
for almost every sufficiently small \(r\).  Consequently,
\begin{equation}
\label{eq:almost-weiss-monotonicity}
 r\longmapsto \mathcal W(r,u\chi;P)+Cr^2
 \quad\text{non-decreasing for }0<r<r_0.
\end{equation}
In particular, \(\lim_{r\downarrow0}\mathcal W(r,u\chi;P)\) exists.  This is the content of
\cite[Proposition~4.1]{Bowman2025}.  If \(u_{P,r_j}\to u_0\), scaling and
\eqref{eq:almost-weiss-monotonicity} show that
\(\mathcal W(\rho,u_0)\) has the same value for every \(\rho>0\).  In the
limit the error has vanished, and equality in the Weiss formula gives
\[
 x\cdot\nabla_xu_0+2t\partial_tu_0-2u_0=0.
\]
This is equivalent to \eqref{eq:parabolic-homogeneity}.

\smallskip
\noindent
\emph{Step 3: the parabolic dichotomy.}
At this point the classification is no longer a genuinely Kolmogorov
argument.  The proof invokes the global parabolic obstacle theory of
\cite{CaffarelliPetrosyanShahgholian2004}, see also
\cite[Proposition~5.6]{Bowman2025}.  A global solution of
\eqref{eq:parabolic-global-obstacle} with quadratic growth is non-increasing
in time and convex in the diffusion variables, i.e.,
\begin{equation}
\label{eq:parabolic-convexity}
 \partial_tu_0\leq0,
 \qquad D_x^2u_0\geq0.
\end{equation}
For a degree-two homogeneous solution, the parabolic classification theorem
then gives precisely \eqref{eq:regular-global-branch} or
\eqref{eq:singular-global-branch}.  Our normalization of the polynomial is
chosen so that
\(\mathcal H(at+\frac12x\cdot Ax)=\operatorname{tr}A-a=1\).  There is also
an energy formulation of the dichotomy.  The balanced energy
at \(P\) is
\begin{equation}
\label{eq:balanced-energy}
 \mathcal E(P;u)
 =\lim_{r\downarrow0}\mathcal W(r,u\chi;P).
\end{equation}
It takes only the two values
\begin{equation}
\label{eq:two-energy-values}
 \mathcal E(P;u)\in\{\omega_*,2\omega_*\},
 \qquad \omega_*>0,
\end{equation}
where \(\omega_*\) is the energy of a stationary half-space solution for
the chosen normalization of \(\mathscr G\).  The value \(\omega_*\)
characterizes the half-space branch and
\(2\omega_*\) the polynomial branch.  Since the corrected energies at
positive scale are continuous in the center and monotone down to
\eqref{eq:balanced-energy}, \(\mathcal E\) is upper semicontinuous on the
free boundary.  Thus the regular set is relatively open.

\smallskip
\noindent
\emph{Step 4: thickness selects the regular branch and fixes its direction.}
The thickness is stable under the compactness above.  For a global parabolic
limit, a lower bound for the minimal diameter of a time slice of the contact
set, together with John's ellipsoid lemma and
\eqref{eq:parabolic-convexity}, produces a ball in the contact set which
persists over a non-trivial time interval.  A non-zero polynomial in
\eqref{eq:singular-global-branch} cannot vanish on such a cylinder.  A
compactness contradiction therefore shows that a one-scale thickness bound,
after the quantitative rescaling used in \cite{Bowman2025}, forces
\(\mathcal E=\omega_*\) at every nearby free-boundary point.  Equivalently,
all blow-ups there belong to \eqref{eq:regular-global-branch}, see
\cite[Lemmas~6.1 and 6.6]{Bowman2025}.

It remains to explain uniqueness.  If one blow-up has direction \(e\), the
improvement-of-minimum lemma and the commutator estimate for diffusion
derivatives give monotonicity of \(u\) in every direction in a cone about
\(e\) at all sufficiently small scales.  Any second half-space blow-up must
inherit all of these directional monotonicities.  Letting the aperture tend
to a hemisphere forces its normal to be \(e\).  Hence the half-space blow-up
is unique, as in \cite[Proposition~6.11]{Bowman2025}.
\end{proof}

\begin{remark}
\label{rem:meaning-global-classification}
It is useful to distinguish two statements sometimes both referred to as
``classification of global solutions.''  The compactness and scaling
argument of \cite[Proposition~5.1]{Bowman2025} first shows that every global
limit obtained by blowing up a free-boundary point is independent of \(y\)
and solves the parabolic obstacle problem.  General global parabolic
solutions with quadratic growth are non-increasing in time and convex in
\(x\).  The sharper half-space/polynomial dichotomy additionally uses the
degree-two homogeneity furnished by the Weiss formula and therefore applies
to global blow-up limits.  No assertion that an arbitrary non-homogeneous
global solution has one of the two explicit forms is used here.
\end{remark}

The classification enters our proof in three concrete ways.  First, the
exclusion of \eqref{eq:singular-global-branch} gives a neighborhood consisting
entirely of regular points.  Second, the half-space approximation supplies
the common monotonicity cone.  Third, low-energy compactness is the input
to the proof in \cite{Bowman2025} that \(Yu\) vanishes continuously on the
regular free boundary.  The latter two facts yield the quantitative
reference-value estimates through the spatial-slice argument below.  Our boundary Harnack inequalities
are used only after these facts have been established and their role is to turn
the qualitative information on the normal and the transport derivative into
power moduli.

An immediate consequence of \eqref{eq:bowman-graph-regularity} is
\begin{equation}
\label{eq:ym-half-holder}
 |\psi(x',y',y_m,t)-\psi(x',y',\widetilde y_m,t)|
 \leq
 C|y_m-\widetilde y_m|^{1/2}.
\end{equation}
We claim that, at Kolmogorov scale \(r\), the corresponding normalized
\(C^{0,1/3}_{y_m}\) deviation is bounded by \(Cr^{1/2}\).  Indeed, after
translating \(P\in\Delta_\psi\) to the origin by the group law, the
rescaled graph is
\begin{equation}
\label{eq:rescaled-graph-at-origin}
 \psi_{P,r}(x',y,t)
 =
 r^{-1}\psi_P(rx',r^3y,r^2t).
\end{equation}
It follows from \eqref{eq:ym-half-holder} that
\begin{equation}
\label{eq:rescaled-ym-half-holder}
 \left|
 \psi_{P,r}(x',y',y_m,t)
 -
 \psi_{P,r}(x',y',\widetilde y_m,t)
 \right|
 \leq
 Cr^{1/2}|y_m-\widetilde y_m|^{1/2}.
\end{equation}
Consequently, on every fixed normalized box,
\[
 \frac{
 \left|
 \psi_{P,r}(x',y',y_m,t)
 -
 \psi_{P,r}(x',y',\widetilde y_m,t)
 \right|}
 {|y_m-\widetilde y_m|^{1/3}}
 \leq
 C_Rr^{1/2},
\]
where \(C_R\) depends only on the normalized radius.  Taking the supremum
over the pairs occurring in
\eqref{eq:localized-cylindrical-defect} gives
\begin{equation}
\label{eq:cylindrical-defect-decay}
 \operatorname{cyl}_{R,P,r}(\psi)
 \leq
 C_Rr^{1/2}.
\end{equation}
Thus the free-boundary graph is asymptotically cylindrical, uniformly in
the regular neighborhood, in precisely the sense required by
Theorem~\ref{thm:BHI-input} below.

\subsection{Homogeneous boundary comparison}

We next record the localized boundary Harnack inequality used in this
paper.  It follows from the boundary comparison theorem for asymptotically
cylindrical intrinsic Lipschitz domains established in
\cite{NystromAsymptoticallyCylindrical2026}, after translation, dilation,
and restriction to a smaller intrinsic box.  The statement is formulated
below in the notation of the present paper.  For the earlier boundary
comparison theory in Lipschitz-type Kolmogorov domains, we refer to
\cite{NystromPolidoro2016}.  The formulation below is tailored to the
present application.

In this subsection,
\[
 \Omega_\psi
 =
 \{(x',x_m,y,t):x_m>\psi(x',y,t)\},
 \qquad
 \Delta_\psi=\partial\Omega_\psi,
\]
denotes a general intrinsic Lipschitz graph domain.  In the application to
the obstacle problem, \(\Omega_\psi\) agrees locally with the positivity set
\(\Omega_u\).  The forward and backward reference points are denoted by
\(A_{r,\Lambda}^{\pm}(P)\), and \(A_{r,\Lambda}(P)\) denotes the central
interior reference point.  More precisely, in coordinates centered at the
origin, set
\begin{align}
 A_{r,\Lambda}^{\pm}
 &=(0',\Lambda r,0',\mp\tfrac23\Lambda r^3,\pm r^2),\quad A_{r,\Lambda}
 =(0',\Lambda r,0',0,0).
\label{eq:reference-points}
\end{align}
where the entries are ordered as \((x',x_m,y',y_m,t)\), and define
\begin{equation}
\label{eq:translated-reference-points}
 A_{r,\Lambda}^{\pm}(P)=P\circ A_{r,\Lambda}^{\pm},
 \qquad A_{r,\Lambda}(P)=P\circ A_{r,\Lambda}.
\end{equation}
For \(\Lambda\) sufficiently large, depending only on the Lipschitz
constant, these are quantitative interior reference points.  The displacement
in \(y_m\) compensates for the transport generated by the displacement in
\(x_m\). This is why Euclidean vertical corkscrew points alone are not
appropriate.

Let \(P\in\Delta_\psi\) and \(r>0\) be such that the reference points
\(A_{r,\Lambda}^{\pm}(P)\) belong to \(\Omega_\psi\).  We say that a positive
function \(v\) is \(H\)-balanced with respect to \(P\) at scale \(r\) if
\begin{equation}
\label{eq:reference-balance-definition}
 H^{-1}
 \leq
 \frac{v(A_{r,\Lambda}^{+}(P))}
      {v(A_{r,\Lambda}^{-}(P))}
 \leq H.
\end{equation}
Thus the forward and backward reference values of \(v\) are quantitatively
comparable, with a comparison constant independent of the scale and the
boundary point in the coordinate patch under consideration.

The quotient estimate below is a scale-\(r\) consequence of
\cite[Theorem~3.1]{NystromAsymptoticallyCylindrical2026}.  The uniform
smallness assumption in \eqref{eq:BHI-defect-hypothesis} is slightly
stronger than the hypothesis at a single center and scale appearing there.
We impose it because the argument below applies boundary comparison after
further recentering and rescaling.

\begin{theorem}
\label{thm:BHI-input}
Fix \(M,H\geq1\), and let \(\Lambda=\Lambda(m,M)\) be the constant used in
the construction of the reference points.  There exists
\(\eta_0=\eta_0(m,M,H)\in(0,1)\) such that the following holds.  Let
\(\Omega_\psi\) be a non-characteristic intrinsic Lipschitz graph domain in
the sense of Definition~\ref{def:intrinsic-lipschitz-domain}, with intrinsic
Lipschitz constant at most \(M\).  Let \(P\in\Delta_\psi\) and \(r>0\), and
suppose that
\begin{equation}
\label{eq:BHI-defect-hypothesis}
 \sup_{\substack{
 \widetilde P\in\Delta_\psi\cap Q_{M,2r}(P)\\
 0<s\leq r}}
 \operatorname{cyl}_{2,\widetilde P,s}(\psi)
 \leq\eta_0,
\end{equation}
with every box occurring in
\eqref{eq:BHI-defect-hypothesis} contained in the graph coordinate patch.
Let \(v_1,v_2>0\) satisfy
\begin{equation}
\label{eq:BHI-harmonic-functions}
 \K v_1=\K v_2=0
 \quad\text{in }\Omega_\psi\cap Q_{M,2r}(P),
\end{equation}
vanish continuously on
\(\Delta_\psi\cap Q_{M,2r}(P)\), and be \(H\)-balanced with respect to
\(P\) at scale \(r\), in the sense of
\eqref{eq:reference-balance-definition}.  Then there exist
\[
 \kappa=\kappa(m,M)\in(0,1),
 \qquad
 \sigma=\sigma(m,M,H)\in(0,1),
\]
and \(C=C(m,M,H)\geq1\) such that
\begin{equation}
\label{eq:homogeneous-holder}
 \left|
 \frac{v_2(Z)}{v_1(Z)}
 -
 \frac{v_2(\widetilde Z)}{v_1(\widetilde Z)}
 \right|
 \leq
 C
 \left(
 \frac{\dK(Z,\widetilde Z)}{r}
 \right)^\sigma
 \frac{v_2(A_{r,\Lambda}(P))}
      {v_1(A_{r,\Lambda}(P))}
\end{equation}
whenever
\[
 Z,\widetilde Z
 \in
 \Omega_\psi\cap Q_{M,\kappa r}(P).
\]
In particular, \(v_2/v_1\) extends continuously to the corresponding graph
portion.
\end{theorem}

We next record the oscillation consequence of
Theorem~\ref{thm:BHI-input}.  Set
\[
 D_r^\psi(P)=\Omega_\psi\cap Q_{M,r}(P)
\]
and suppose that
\[
 m_r
 =
 \inf_{D_r^\psi(P)}\frac{v_2}{v_1},
 \qquad
 M_r
 =
 \sup_{D_r^\psi(P)}\frac{v_2}{v_1}
\]
are finite.  Define
\begin{equation}
\label{eq:homogeneous-oscillation-envelopes}
 E_r^-
 =
 v_2-m_rv_1,
 \qquad
 E_r^+
 =
 M_rv_1-v_2.
\end{equation}
These are, respectively, the lower and upper oscillation envelopes of
\(v_2/v_1\) at scale \(r\).  They are nonnegative \(\K\)-harmonic
functions in \(D_r^\psi(P)\), vanish continuously on the graph portion,
and satisfy
\[
 \frac{E_r^-}{v_1}
 =
 \frac{v_2}{v_1}-m_r,
 \qquad
 \frac{E_r^+}{v_1}
 =
 M_r-\frac{v_2}{v_1}.
\]

Suppose that \(v_1\) and one of the nontrivial envelopes in
\eqref{eq:homogeneous-oscillation-envelopes} are
\(H_{\mathrm{env}}\)-balanced with respect to \(P\) at scale \(r/2\), and
that \eqref{eq:BHI-defect-hypothesis} holds with
\(\eta_0=\eta_0(m,M,H_{\mathrm{env}})\).  Then there exist
\[
 \theta=\theta(m,M,H_{\mathrm{env}})\in(0,1),
 \qquad
 \vartheta=\vartheta(m,M,H_{\mathrm{env}})\in(0,1),
\]
such that
\begin{equation}
\label{eq:homogeneous-contraction}
 \osc_{D_{\theta r}^\psi(P)}
 \frac{v_2}{v_1}
 \leq
 \vartheta
 \osc_{D_r^\psi(P)}
 \frac{v_2}{v_1}.
\end{equation}
Indeed, put \(\omega_r=M_r-m_r\), and let \(E_r\) be the selected
balanced envelope.  For \(\tau>0\), apply Theorem~\ref{thm:BHI-input}
at scale \(r/2\) to \(v_1\) and \(E_r+\tau v_1\).  The latter is
strictly positive and remains \(H_{\mathrm{env}}\)-balanced.  Use
\(0\leq E_r/v_1\leq\omega_r\), let \(\tau\downarrow0\), and choose
\(\theta\) sufficiently small in \eqref{eq:homogeneous-holder}.

If \(\omega_r>0\) and an envelope has poorly balanced reference values,
adding \(\delta\omega_rv_1\), for fixed \(\delta>0\), supplies the
required balance without changing its quotient oscillation.  We carry
out this adjustment explicitly in the proof of
Proposition~\ref{prop:inhomogeneous-contraction}.

\subsection{Uniform reference values}
\label{subsec:uniform-reference-values}

We now use cone monotonicity and the smallness of \(Yu\) near the regular
graph to obtain the uniform reference-point estimates required for boundary
comparison.  The argument uses the equation on spatial slices and does not
require a quantitative modulus for the blow-up normal.

\begin{lemma}
\label{lem:uniform-reference-values}
Let \(\mathcal C'\subset\mathbb S^{m-1}\) be a closed spherical cone
compactly contained in the common monotonicity cone from
Theorem~\ref{thm:bowman-input}\textnormal{(v)}.  Then there exist
\(r_1\in(0,r_*]\) and constants \(c_1,C_1>0\), depending only on the
fixed regular neighborhood and on \(\mathcal C'\), and independent of
\(P\in\Gamma_u\cap Q_{M,2r_*}(P_0)\), such that
\begin{equation}
\label{eq:uniform-reference-values}
 c_1
 \leq
 \partial_eu_{P,r}(A_{1,\Lambda}^{\pm})
 \leq
 C_1
\end{equation}
for every \(0<r\leq r_1\), every \(e\in\mathcal C'\), and both choices of
sign.  Consequently, these directional derivatives are uniformly balanced
at the forward and backward reference points.  Their reference values are
also mutually comparable, uniformly with respect to \(P\), \(r\), \(e\),
and the choice of sign.  In addition, on every fixed rescaled working box,
after reducing \(r_1\) if necessary,
\begin{equation}
\label{eq:cone-monotonicity-input}
 0\leq u_{P,r}\leq C\partial_eu_{P,r},
 \qquad e\in\mathcal C',\quad 0<r\leq r_1.
\end{equation}
In the last assertion, \(C\) and the admissible choice of \(r_1\) may also
depend on the fixed rescaled working box.  The choice of \(r_1\) includes
dependence on the uniform modulus \(\omega\) of boundary vanishing of
\(Yu\) on the fixed regular neighborhood.
\end{lemma}

\begin{proof}
To prove the lemma we will use the equation on fixed \((y,t)\)-slices, together with the continuous
vanishing of \(Yu\) in \eqref{eq:Yu-boundary-vanishing}.  Write \(\mathcal C\) for the open common monotonicity cone.
For fixed \(P\) and \(r\), work with the intrinsically translated
and rescaled solution
\[
 v(Z)=u_{P,r}(Z)=r^{-2}u(P\circ\delta_r Z).
\]
In these coordinates, \(P\) corresponds to the origin and the
intrinsic scale \(r\) becomes unit scale. Set
\[
 x_A=\Lambda e_m,\qquad
 (y_\pm,t_\pm)
 =\bigl(\mp\tfrac23\Lambda e_m,\pm1\bigr).
\]
Thus \(A_{1,\Lambda}^{\pm}=(x_A,y_\pm,t_\pm)\).

The quantitative reference-point geometry and the uniform intrinsic
Lipschitz character give constants \(a>0\) and \(L<\infty\), independent
of \(P\), \(r\), and the sign, such that first
\[
 \overline{B_{4a}(x_A)}\times\{(y_\pm,t_\pm)\}
 \subset\Omega_v
\]
and these spatial balls lie in the fixed rescaled region where cone
monotonicity holds.  Here and below the balls are Euclidean balls in the
\(x\)-variables.  Second, if
\[
 b_\pm=\psi_{P,r}(0',y_\pm,t_\pm),
\]
then \(0<\Lambda-b_\pm\leq L\).  By initially restricting \(r\) to a
sufficiently small fixed interval, all the corresponding unscaled balls
and the vertical segments from \((0',b_\pm,y_\pm,t_\pm)\) to the reference
points lie in \(Q_{M,3r_*}(P_0)\).  This choice is uniform because the
closure of the indicated set of centers is contained in the larger
regular coordinate patch.

The Hessian estimate in \eqref{eq:bowman-solution-bounds} is invariant
under the rescaling defining the function \(v\).  The spatial \(C^{1,1}\) regularity and
nonnegativity of \(v\) imply that its spatial gradient vanishes at every
point of its zero set.  Integration along the preceding vertical segment
therefore gives
\[
 |\nabla_xv(x_A,y_\pm,t_\pm)|
 \leq C(\Lambda-b_\pm)
 \leq CL.
\]
This proves the upper bound in \eqref{eq:uniform-reference-values}.

For the lower bound, compact containment of \(\mathcal C'\) in
\(\mathcal C\) gives a number \(\delta\in(0,1)\) such that
\[
 {(e-\delta\xi)}/{|e-\delta\xi|}\in\mathcal C
 \qquad
 \text{for every }e\in\mathcal C'
 \text{ and every }\xi\in\mathbb S^{m-1}.
\]
By \eqref{eq:cone-monotonicity-sign}, all directional derivatives in
\(\mathcal C\) are nonnegative in the working region.  At any point
where \(\nabla_xv\neq0\), apply this to
\(\xi=\nabla_xv/|\nabla_xv|\).  It follows, also trivially when the
gradient vanishes, that
\[
 \delta|\nabla_xv|\leq\partial_ev
 \qquad\text{for every }e\in\mathcal C'.
\]

Let \(\varepsilon\in(0,1/2]\) be chosen below.  Left invariance and
homogeneity give
\[
 Yv(Z)=(Yu)(P\circ\delta_rZ).
\]
By \cite[Proposition~6.7]{Bowman2025}, with the localization described
after Theorem~\ref{thm:bowman-input}, \(Yu\) vanishes continuously on
the regular free boundary. Since the closure of the set of centers
is compactly contained in the larger regular neighborhood, there is
a nondecreasing modulus \(\omega\), with
\(\omega(s)\to0\) as \(s\downarrow0\), such that
\[
 |Yu(Q)|\leq\omega(\dK(P,Q))
\]
uniformly for the boundary points \(P\) under consideration and
nearby points \(Q\in\Omega_u\). The unscaled points in the fixed working boxes have intrinsic distance
at most \(Cr\) from their respective centers \(P\), with \(C\)
independent of \(P\) and \(r\). Choose \(r_1\in(0,r_*]\) sufficiently
small that these boxes remain in the regular neighborhood and $\omega(Cr_1)\leq\varepsilon$. The rescaling identity for \(Yv\) therefore gives
\[
 |Yv(x,y_\pm,t_\pm)|\leq\varepsilon
 \qquad\text{for }x\in B_{4a}(x_A),\quad 0<r\leq r_1,
\]
uniformly in \(P\) and the sign. The initial radius \(r_1\) may depend
on \(\omega\). Only smallness below a fixed threshold is needed,
and no quantitative rate of vanishing is required.

Fix a sign and a direction \(e\in\mathcal C'\), and put
\[
 V(x)=v(x,y_\pm,t_\pm),\qquad
 F(x)=Yv(x,y_\pm,t_\pm),\qquad
 q(x)=\partial_eV(x).
\]
The solution is smooth in its positivity set by interior
hypoellipticity.  On the spatial ball \(B_{4a}(x_A)\), its equation gives
\begin{equation}
\label{eq:reference-slice-equation}
 \Delta_xV=1-F,\qquad
 \Delta_xq=-\partial_eF,\qquad
 \|F\|_{\mathrm L^\infty(B_{4a}(x_A))}\leq\varepsilon.
\end{equation}
In particular, \(\Delta_xV\geq1/2\) and
\(q\geq\delta|\nabla_xV|\geq0\).  With \(n_x\) denoting the outer
unit normal to \(B_a(x_A)\), the divergence theorem yields
\[
 \frac12|B_a|
 \leq\int_{B_a(x_A)}\Delta_xV\,\d x
 =\int_{\partial B_a(x_A)}\nabla_xV\cdot n_x\,\mathrm dS_x
 \leq\delta^{-1}|\partial B_a|
       \sup_{B_a(x_A)}q.
\]
Since \(|B_a|/|\partial B_a|=a/m\), this implies
\[
 \sup_{B_a(x_A)}q\geq\frac{\delta a}{2m}.
\]

To transfer this lower bound to the center, let \(h\) be the harmonic
function in \(B_{2a}(x_A)\) with boundary values \(q\).  The maximum
principle gives \(h\geq0\).  If \(G_{2a}\) is the Dirichlet Green
function for \(-\Delta_x\) in this ball, then
\eqref{eq:reference-slice-equation} and integration by parts give
\[
 q(x)-h(x)
 =-\int_{B_{2a}(x_A)}
       \partial_{e,\xi}G_{2a}(x,\xi)F(\xi)\,\d\xi.
\]
The elementary Green-function estimate for a ball,
\[
 \sup_{x\in B_{2a}(x_A)}
 \int_{B_{2a}(x_A)}
       |\nabla_\xi G_{2a}(x,\xi)|\,\d\xi
 \leq C_ma,
\]
therefore yields
\[
 \|q-h\|_{\mathrm L^\infty(B_{2a}(x_A))}
 \leq C_ma\varepsilon.
\]
The interior Harnack inequality for the nonnegative harmonic function
\(h\) now gives constants \(H_m,C_m'>0\), depending only on dimension,
such that
\[
 \frac{\delta a}{2m}
 \leq\sup_{B_a(x_A)}q
 \leq H_mh(x_A)+C_ma\varepsilon
 \leq H_mq(x_A)+C_m'a\varepsilon.
\]
Choose \(\varepsilon\leq\delta/(4mC_m')\), in addition to
\(\varepsilon\leq1/2\), and then fix \(r_1\) as above.  We conclude that
\[
 \partial_ev(A_{1,\Lambda}^{\pm})
 =q(x_A)
 \geq\frac{\delta a}{4mH_m}>0.
\]
Every constant and the choice of \(r_1\) are uniform in the center, the
direction, and the sign.  Thus \eqref{eq:uniform-reference-values} holds
with \(c_1=\delta a/(4mH_m)\) and the upper constant already obtained.
The balance and mutual comparability assertions follow by taking ratios
of these bounds.

For completeness, the same slice argument proves
\eqref{eq:cone-monotonicity-input}.  Let \(Z=(x,y,t)\) be a point of the
rescaled positivity set in the fixed working box, and put
\(\ell=x_m-\psi_{P,r}(x',y,t)>0\).  The spatial Lipschitz bound for the graph
gives \(a_*>0\), depending only on its Lipschitz constant, such that
\(B_{4a_*\ell}(x)\times\{(y,t)\}\) lies in the positivity set and in a fixed
enlargement of the working box.  Choose \(r_1\) so that \(|Yv|\leq\varepsilon\)
also on this enlargement.  Repeating the preceding divergence-theorem and
Green-function argument with radius \(a_*\ell\) gives
\(\partial_ev(Z)\geq c\ell\), uniformly for \(e\in\mathcal C'\).
On the other hand, integrating the spatial Hessian bound from the graph
point directly below \(Z\), where \(v=|\nabla_xv|=0\), gives
\(v(Z)\leq C\ell^2\).  Since \(\ell\) is bounded above on a fixed working box,
these inequalities imply \(v(Z)\leq C\partial_ev(Z)\).
The estimate is also immediate on the contact set, and proves
\eqref{eq:cone-monotonicity-input}.
\end{proof}

\begin{remark} Since
\[
 A_{r,\Lambda}^{\pm}(P)
 =P\circ\delta_r A_{1,\Lambda}^{\pm},
 \qquad
 \partial_eu_{P,r}(A_{1,\Lambda}^{\pm})
 =r^{-1}\partial_eu(A_{r,\Lambda}^{\pm}(P)),
\]
the bounds in \eqref{eq:uniform-reference-values} are equivalent to
\[
 c_1r
 \leq \partial_eu(A_{r,\Lambda}^{\pm}(P))
 \leq C_1r.
\]
Thus the directional derivatives are uniformly of order \(r\) at the
reference points. At each fixed scale, these values are mutually
comparable, with constants independent of \(P\), \(r\), \(e\), and the
choice of sign. In particular, the forward and backward reference
values are uniformly balanced. Similarly, \eqref{eq:cone-monotonicity-input} becomes
\[
 0\leq u\leq Cr\,\partial_eu
\]
on the corresponding working box in the original variables.
\end{remark}

We now verify the geometric hypothesis of
Theorem~\ref{thm:BHI-input}.  After replacing \(r_*\) by
\(\min\{r_*,r_1\}\), passing, if necessary, to a smaller regular patch,
and reducing \(r_*\) further, we may assume that every box
\(Q_{M,2s}(\widetilde P)\), with
\(\widetilde P\in\Gamma_u\cap Q_{M,2r}(P)\) and \(0<s\leq r\), remains in
the coordinate neighborhood on which
\eqref{eq:cylindrical-defect-decay} holds whenever \(P\) belongs to the
smaller patch and \(0<r\leq r_*\).  The uniformity of
\eqref{eq:cylindrical-defect-decay} with respect to the center then gives
\[
 \operatorname{cyl}_{2,\widetilde P,s}(\psi)
 \leq Cs^{1/2}
\]
for every
\[
 \widetilde P\in\Gamma_u\cap Q_{M,2r}(P),
 \qquad
 0<s\leq r.
\]
Consequently, the deviation quantity required in
Theorem~\ref{thm:BHI-input} satisfies
\begin{equation}
\label{eq:uniform-defect-smallness}
 \mathfrak D(P,r)
 :=
 \sup_{\substack{
 \widetilde P\in\Gamma_u\cap Q_{M,2r}(P)\\
 0<s\leq r}}
 \operatorname{cyl}_{2,\widetilde P,s}(\psi)
 \leq Cr^{1/2}
 \leq Cr_*^{1/2}.
\end{equation}

Let \(H_{\mathrm{ref}}\geq1\) be a uniform balance constant supplied by
Lemma~\ref{lem:uniform-reference-values}.  Fix the parameter
\(\delta\in(0,1/8)\) used in the envelope construction in the proof of
Proposition~\ref{prop:inhomogeneous-contraction} below, and choose \(H_*\geq1\)
so that
\[
 H_*
 \geq
 3H_{\mathrm{ref}}{(1+\delta)}/{\delta}.
\]
This choice of \(H_*\) bounds the balance constants of the reference
derivatives.  The comparison estimates established in
Proposition~\ref{prop:inhomogeneous-contraction} below show that, after a
further reduction of \(r_*\), it also bounds the balance constants of the
\(\K\)-harmonic replacements and the adjusted envelopes used below.
In particular, \(H_*\) depends only on the quantitative data of the fixed
regular neighborhood and is independent of the boundary point, the scale,
and the particular functions under consideration.

After reducing \(r_*\) once more, without changing its notation, we may
assume that
\begin{equation}
\label{eq:defect-small}
 Cr_*^{1/2}
 \leq
 \eta_0(m,M,H_*).
\end{equation}
Together with \eqref{eq:uniform-defect-smallness}, this shows that the
geometric smallness hypothesis of Theorem~\ref{thm:BHI-input} holds
uniformly at every point of the smaller regular patch and every scale
\(0<r\leq r_*\).  Since any function that is \(H\)-balanced for some
\(H\leq H_*\) is also \(H_*\)-balanced, all boundary Harnack constants may
be chosen independently of the boundary point, the scale, and the
particular pair of functions.

\section{Derivative equations and non-degenerate cones}
\label{sec:derivatives}

We first identify the equations satisfied by the derivatives that determine
the intrinsic normal.  Although these derivatives vanish on the free
boundary and are positive in suitable cones, they are not
\(\K\)-harmonic.  The purpose of this section is to quantify the resulting
error under blow-up and to isolate a uniformly non-degenerate family of
directional derivatives.

In \(\Omega_u\), where \(\K u=1\), the commutator identity
\([\K,\partial_{x_i}]=-\partial_{y_i}\) gives
\begin{equation}
\label{eq:derivative-equations}
 \K u_{x_i}=-u_{y_i},
 \qquad i=1,\ldots,m.
\end{equation}
The source terms are bounded by
\eqref{eq:bowman-solution-bounds}.  Recall that the blow-up of \(u\) at
\(P\in\Gamma_u\) and scale \(r>0\) is
\begin{equation}
\label{eq:blowup}
 u_{P,r}(Z)
 =
 r^{-2}u(P\circ\delta_r Z).
\end{equation}
The scaling of the \(y\)-derivatives gives
\[
 (u_{P,r})_{y_i}(Z)
 =
 r\,u_{y_i}(P\circ\delta_r Z).
\]
Consequently, in the rescaled positivity set,
\begin{equation}
\label{eq:scaled-derivative-equations}
 \K (u_{P,r})_{x_i}
 =
 -(u_{P,r})_{y_i}
 =
 -r\,u_{y_i}(P\circ\delta_r Z),
 \qquad i=1,\ldots,m.
\end{equation}
Thus, on every fixed rescaled working box contained in the regular
neighborhood,
\begin{equation}
\label{eq:scaled-derivative-source-bound}
 \|\K (u_{P,r})_{x_i}\|_{\mathrm L^\infty}
 +
 \|(u_{P,r})_{y_i}\|_{\mathrm L^\infty}
 \leq Cr.
\end{equation}
In particular, the failure of the diffusion derivatives to be
\(\K\)-harmonic is of lower order under the obstacle-problem blow-up.

The intrinsic first-order regularity established in
\cite{Bowman2025} implies that \(\nabla_xu\) has a continuous
representative in the intrinsic variables.  Let
\(P=(x_0,y_0,t_0)\in\Gamma_u\).  Since \(u\geq0\) and
\(u(x_0,y_0,t_0)=0\), the differentiable function
\(x\mapsto u(x,y_0,t_0)\) has a local minimum at \(x_0\).  It follows that
\begin{equation}
\label{eq:gradient-zero-free-boundary}
 \nabla_xu(P)=0.
\end{equation}
Hence every diffusion derivative extends continuously by zero to the free
boundary.  The following lemma gives a quantitative form of this
vanishing.

\begin{lemma}
\label{lem:gradient-vanishing}
There exists \(C\geq 1\) such that, for every
\(Z\in\Omega_u\cap Q_{M,r_*}(P_0)\), one has
\begin{equation}
\label{eq:gradient-linear-upper}
 |\nabla_xu(Z)|
 \leq C\,\dK(Z,\Gamma_u).
\end{equation}
Hence, every derivative \(u_{x_i}\) extends continuously to
the regular free boundary by setting \(u_{x_i}=0\) on \(\Gamma_u\).
\end{lemma}

\begin{proof}
Set \(d=\dK(Z,\Gamma_u)\), and let \(P\in\Gamma_u\) be the graph point
having the same \((x',y,t)\)-coordinates as \(Z\).  Write
\(\ell=x_m(Z)-x_m(P)>0\).  The vertical-distance comparison following
from \eqref{eq:intrinsic-lipschitz-condition} gives \(\ell\leq Cd\).
After reducing the coordinate neighborhood if necessary, the vertical
segment from \(P\) to \(Z\) lies in the larger neighborhood where
\eqref{eq:bowman-solution-bounds} holds.  Since
\(\nabla_xu(P)=0\), the spatial \(C^{1,1}\) estimate gives directly
\[
 |\nabla_xu(Z)|
 =|\nabla_xu(Z)-\nabla_xu(P)|
 \leq \|D_x^2u\|_{\mathrm L^\infty(Q_{M,3r_*}(P_0))}\ell
 \leq Cd.
\]
This proves \eqref{eq:gradient-linear-upper} and the asserted continuous
vanishing at the free boundary.
\end{proof}

We work in the coordinates fixed in
Theorem~\ref{thm:bowman-input}.  Choose \(a_0\in(0,1)\), depending only on
the aperture of the common monotonicity cone, so that the directions
\begin{equation}
\label{eq:cone-directions}
 e_j^\pm
 =
 {(e_m\pm a_0e_j)}/{\sqrt{1+a_0^2}},
 \qquad j=1,\ldots,m-1,
\end{equation}
together with \(e_m\), belong to a closed spherical cone compactly
contained in the common monotonicity cone.  Here
\(e_1,\ldots,e_m\) are the standard coordinate vectors in the
\(x\)-variables.  Define
\begin{equation}
\label{eq:positive-combinations}
 q_0:=u_{x_m},
 \qquad
 q_j^\pm:=u_{x_m}\pm a_0u_{x_j},
 \qquad j=1,\ldots,m-1.
\end{equation}
Since
\[
 q_0=\partial_{e_m}u,
 \qquad
 q_j^\pm
 =
 \sqrt{1+a_0^2}\,\partial_{e_j^\pm}u,
\]
these functions are the directional derivatives associated with the fixed
subcone chosen above.  By Lemma~\ref{lem:gradient-vanishing}, they extend
continuously by zero to the free-boundary graph.

For \(P\) in the chosen free-boundary patch and \(0<r\leq r_1\), define
the corresponding derivatives of the rescaled solution by
\begin{equation}
\label{eq:scaled-positive-derivatives}
 q_{0,P,r}:=(u_{P,r})_{x_m},
 \qquad
 q_{j,P,r}^\pm
 :=
 (u_{P,r})_{x_m}\pm a_0(u_{P,r})_{x_j},
 \qquad j=1,\ldots,m-1,
\end{equation}
and set
\begin{equation}
\label{eq:scaled-derivative-family}
 \mathscr F_{P,r}
 :=
 \left\{
 q_{0,P,r},q_{j,P,r}^+,q_{j,P,r}^-:
 j=1,\ldots,m-1
 \right\}.
\end{equation}
The scaling identity
\[
 (u_{P,r})_{x_i}(Z)
 =
 r^{-1}u_{x_i}(P\circ\delta_rZ)
\]
relates the functions in \(\mathscr F_{P,r}\) to those in
\eqref{eq:positive-combinations}.

Applying \eqref{eq:cone-monotonicity-input} to \(e_m\) and \(e_j^\pm\),
and absorbing the factor \(\sqrt{1+a_0^2}\) into the constant, gives
\begin{equation}
\label{eq:u-controlled-by-q}
 0<u_{P,r}\leq Cq,
 \qquad q\in\mathscr F_{P,r},
\end{equation}
in the positivity portion of the fixed rescaled working box.
After reducing the regular free-boundary patch, this box and the constant
\(C\) can be chosen independently of \(P\) and \(r\).  In particular,
every function in \(\mathscr F_{P,r}\) is strictly positive in that
portion of the rescaled positivity set.
Scaling back shows that the functions in
\eqref{eq:positive-combinations} are strictly positive in the corresponding
local positivity region and that
\[
 u\leq Cr\widetilde q,
\]
where \(\widetilde q\) is the corresponding directional derivative of the
unscaled solution.

Lemma~\ref{lem:uniform-reference-values}, applied to the normalized
directions \(e_m\) and \(e_j^\pm\), gives constants
\(c_{\mathrm{ref}},C_{\mathrm{ref}}>0\) such that
\begin{equation}
\label{eq:scaled-reference-value-bounds}
 c_{\mathrm{ref}}
 \leq
 q(A_{1,\Lambda}^{\varepsilon})
 \leq
 C_{\mathrm{ref}},
 \qquad
 q\in\mathscr F_{P,r},
 \qquad
 \varepsilon\in\{+,-\}.
\end{equation}
Here the fixed factor \(\sqrt{1+a_0^2}\) has been absorbed into the
constants.  The bounds are uniform with respect to \(P\) in the chosen
free-boundary patch and \(0<r\leq r_1\).  In particular, with
\(H=C_{\mathrm{ref}}/c_{\mathrm{ref}}\), every
\(q\in\mathscr F_{P,r}\) is \(H\)-balanced with respect to the origin
at scale \(1\),
\begin{equation}
\label{eq:reference-balance-derivatives}
 H^{-1}
 \leq
 \frac{q(A_{1,\Lambda}^{+})}
      {q(A_{1,\Lambda}^{-})}
 \leq H.
\end{equation}
Moreover, \eqref{eq:scaled-reference-value-bounds} shows directly that all
reference values associated with functions in \(\mathscr F_{P,r}\) are
mutually comparable, with constants depending only on the quantitative
data of the regular free-boundary neighborhood.

\section{Inhomogeneous boundary oscillation}
\label{sec:inhomogeneous}

We now establish the perturbative oscillation estimate that transfers the
homogeneous boundary Harnack inequality to the inhomogeneous derivative
equations \eqref{eq:derivative-equations}.  The key step is to replace the
directional derivatives by \(\K\)-harmonic functions and to control the
replacement error using the scale decay in
\eqref{eq:scaled-derivative-source-bound}.

Unless otherwise stated, a weak or energy solution of
\begin{equation}
\label{eq:general-inhomogeneous-equation}
 \K w=f
 \quad\text{in an open set }G
\end{equation}
is a function
\[
 w\in\mathrm L^2_{\mathrm{loc}}(G),
 \qquad
 \nabla_xw\in\mathrm L^2_{\mathrm{loc}}(G),
\]
such that
\begin{equation}
\label{eq:energy-solution-definition}
 \int_G
 \left[
 \nabla_xw\cdot\nabla_x\varphi
 +
 w\bigl(x\cdot\nabla_y\varphi-\partial_t\varphi\bigr)
 \right]
 \,\d x\d y\d t
 =
 -\int_G f\varphi\,\d x\d y\d t
\end{equation}
for every \(\varphi\in C_0^\infty(G)\).  When \(f=0\), we refer to \(w\)
as a \(\K\)-harmonic energy solution, or simply \(\K\)-harmonic.  The
formulation
\eqref{eq:energy-solution-definition} is equivalent to
\(\K w=f\) in the sense of distributions.  In particular, a
\(\K\)-harmonic energy solution is smooth in the interior by
hypoellipticity.  Energy subsolutions and supersolutions are defined by the
corresponding distributional inequalities.  With our sign convention,
\(\K w\geq f\) means that the left-hand side of
\eqref{eq:energy-solution-definition} is at most
\[
 -\int_G f\varphi\,\d x\d y\d t
\]
for every nonnegative \(\varphi\in C_0^\infty(G)\), the inequality is
reversed when \(\K w\leq f\).

In the Dirichlet arguments below, the prescribed data are restrictions
of continuous functions on the closure of the working domain.  We use the
Perron solution and its continuous attainment of these data on the regular
graph portion.  The following convention makes precise the boundary
conditions needed for $\mathcal K$-harmonic replacement.

\subsection{The Kolmogorov boundary and the Dirichlet problem}
\label{subsec:Kolmogorov-Dirichlet}

Since \(\K\) is directed in time, a solution need not attain arbitrary
continuous data on the entire topological boundary of a space-time domain.
For a bounded open set \(G\) and \(g\in C(\partial G)\), let \(H_Gg\)
denote the Perron-Wiener solution for \(\K\).  Continuous boundary data
are resolutive, so this is a well-defined bounded \(\K\)-harmonic function.
The map \(g\mapsto H_Gg\) is linear, preserves order, and preserves
constants, see \cite[Sections~3-4]{Kogoj2017}.

A boundary point \(W\in\partial G\) is regular if
\(H_Gg(Z)\to g(W)\) as \(G\ni Z\to W\) for every \(g\in C(\partial G)\).
In this paper we use the notation
\begin{equation}
\label{eq:Kolmogorov-boundary}
 \partial_{\K}G
 =
 \{W\in\partial G:W\text{ is regular for the Dirichlet problem for }\K\}.
\end{equation}
Thus boundary conditions written on \(\partial_{\K}G\) refer to
pointwise attainment at regular points.  The solution itself is selected
by the Perron construction with the specified continuous data on
\(\partial G\).

We apply this convention to bounded intrinsic graph truncations centered
at graph-boundary points,
\begin{equation}
\label{eq:bounded-box-truncation}
 G=\Omega_\psi\cap Q_{M,\rho}(P),
 \qquad
 P\in\Delta_\psi,\qquad 1\leq\rho<2.
\end{equation}
In particular, this class includes the rescaled positivity-set
truncations introduced in the next subsection.  Every point of the graph
portion \(\Delta_\psi\cap Q_{M,\rho}(P)\) is regular. The intrinsic
Lipschitz geometry supplies the exterior cones used in the barrier
criterion, see \cite[Section~6]{LitsgardNystrom2022}.  The initial time
face and the non-characteristic diffusion faces are also regular away
from their edges.  On the artificial \(y\)-faces, the direction of the
drift determines which portions admit pointwise Dirichlet data.  We do not
require a classification of the artificial corners or of the characteristic
parts of these faces.

For \(g\in C(\overline G)\), the notation
\begin{equation}
\label{eq:Kolmogorov-Dirichlet-problem}
 \begin{cases}
  \K h=0 & \text{in }G,\\
  h=g    & \text{on }\partial_{\K}G
 \end{cases}
\end{equation}
will always mean \(h=H_G(g|_{\partial G})\).  This function is smooth in
\(G\), is a local energy solution, and attains \(g\) continuously on the
graph portion.  This is the only boundary continuity required below.

We record the comparison property in the form used in the proofs.  If
\(z\in C(\overline G)\) is an energy subsolution and
\begin{equation}
\label{eq:weak-comparison-hypothesis}
 \K z\geq0\quad\text{in }G,
 \qquad
 z\leq g\quad\text{on }\partial G,
\end{equation}
then the Perron comparison principle gives
\begin{equation}
\label{eq:weak-comparison-conclusion}
 z\leq H_G(g|_{\partial G})\quad\text{in }G.
\end{equation}
For a continuous energy supersolution with \(\K z\leq0\) and
\(z\geq g\) on \(\partial G\), the inequality is reversed.
In particular, continuous subsolutions and supersolutions that bound the
prescribed data on \(\partial G\) also bound their $\mathcal K$-harmonic replacement
throughout \(G\).  Order preservation gives
\(\|H_Gg\|_{\mathrm L^\infty(G)}\leq\|g\|_{\mathrm L^\infty(\partial G)}\).

\subsection{\texorpdfstring{\(\K\)-harmonic}{K-harmonic} replacement}
\label{subsec:harmonic-replacement}

For later use, define the truncated positivity domain
\begin{equation}
\label{eq:truncated-positivity-domain}
 D_s(P)
 =
 \Omega_u\cap Q_{M,s}(P).
\end{equation}
If \(P\in\Gamma_u\) and \(r>0\), let
\begin{equation}
\label{eq:rescaled-positivity-domain}
 \Omega_{P,r}
 =
 \{u_{P,r}>0\},
 \qquad
 D_s^{P,r}
 =
 \Omega_{P,r}\cap Q_{M,s}(0).
\end{equation}
These domains are related by
\begin{equation}
\label{eq:rescaled-truncation-identity}
 D_s^{P,r}
 =
 \delta_{1/r}\bigl(P^{-1}\circ D_{rs}(P)\bigr).
\end{equation}
For \(P\) in the regular free-boundary patch and \(r\leq r_*\), the set
\(\Omega_{P,r}\) is an intrinsic graph domain.  Hence
\(D_s^{P,r}\) is one of the bounded graph truncations covered by
Subsection~\ref{subsec:Kolmogorov-Dirichlet}.  By the translation and
dilation invariance of \(\K\), together with the
normalization in \eqref{eq:blowup}, we have
\begin{equation}
\label{eq:rescaled-obstacle-equation}
 \K u_{P,r}=1
 \qquad\text{in }\Omega_{P,r}.
\end{equation}

\begin{lemma}
\label{lem:harmonic-replacement}
Let \(P\) belong to the chosen regular free-boundary patch, let
\(0<r\leq r_*\), and let \(\rho\in[1,2)\).  We work in the coordinates
centered at \(P\) and rescaled by \(r\), and set
\begin{equation}
\label{eq:harmonic-replacement-domain}
 G
 =
 D_\rho^{P,r}
 =
 \Omega_{P,r}\cap Q_{M,\rho}(0).
\end{equation}
Suppose that \(w,q_0\in C(\overline G)\) are energy solutions satisfying
\(q_0>0\) in \(G\), and
\begin{equation}
\label{eq:q-equation}
 \K w=f,
 \qquad
 \K q_0=f_0
 \qquad\text{in }G,
\end{equation}
where, for some \(\eps\geq0\),
\begin{equation}
\label{eq:q-source-smallness}
 \|f\|_{\mathrm L^\infty(G)}
 +
 \|f_0\|_{\mathrm L^\infty(G)}
 \leq\eps.
\end{equation}
Assume also that
\begin{equation}
\label{eq:barrier-domination}
 0\leq u_{P,r}\leq C_0q_0
 \qquad\text{in }G.
\end{equation}
Let \(w^h=H_G(w|_{\partial G})\) be the Perron solution of
\begin{equation}
\label{eq:harmonic-replacement-problem}
 \begin{cases}
  \K w^h=0 & \text{in }G,\\
  w^h=w    & \text{on }\partial_{\K}G,
 \end{cases}
\end{equation}
with the convention in Subsection~\ref{subsec:Kolmogorov-Dirichlet}.
Then there exist
\[
 \eps_0=\eps_0(C_0)>0,
 \qquad
 C=C(C_0)<\infty,
\]
such that, if \(0\leq\eps\leq\eps_0\), then
\begin{equation}
\label{eq:harmonic-replacement-error}
 |w-w^h|
 \leq
 C\eps q_0
 \qquad\text{in }G.
\end{equation}
In particular, if \(w=q_0\) and \(q_0^h\) denotes the
\(\K\)-harmonic replacement of \(q_0\), then
\begin{equation}
\label{eq:replacement-relative-error}
 |q_0-q_0^h|
 \leq
 C\eps q_0
 \qquad\text{in }G.
\end{equation}
Consequently, if \(C\eps\leq1/2\), then
\begin{equation}
\label{eq:replacement-comparable}
 \frac12q_0
 \leq
 q_0^h
 \leq
 \frac32q_0
 \qquad\text{in }G.
\end{equation}
\end{lemma}

\begin{proof}
Set \(v=w-w^h\).  Choose \(A>C_0\), for example \(A=C_0+1\), and define
\begin{equation}
\label{eq:harmonic-replacement-barrier}
 B=Aq_0-u_{P,r}.
\end{equation}
By \eqref{eq:barrier-domination},
\begin{equation}
\label{eq:barrier-size}
 0\leq B\leq Aq_0
 \qquad\text{in }G.
\end{equation}
Moreover, since \(G\subset\Omega_{P,r}\), equations
\eqref{eq:rescaled-obstacle-equation} and \eqref{eq:q-equation} give
\begin{equation}
\label{eq:barrier-equation}
 \K B=Af_0-1.
\end{equation}
Choose \(\eps_0(C_0)\leq1/(2A)\).  It follows from
\eqref{eq:q-source-smallness} that
\begin{equation}
\label{eq:barrier-superharmonicity}
 \K B\leq-1/2
 \qquad\text{in }G.
\end{equation}
Using \(|f|\leq\eps\) and
\eqref{eq:barrier-superharmonicity}, we obtain
\begin{align}
 \K(w-2\eps B)
 &=f-2\eps\K B
 \geq-\eps+\eps=0,
\label{eq:upper-comparison-equation}\\
 \K(w+2\eps B)
 &=f+2\eps\K B
 \leq\eps-\eps=0.
\label{eq:lower-comparison-equation}
\end{align}
The functions \(w\pm2\eps B\) are continuous on \(\overline G\), and
\(B\geq0\) there.  Hence they bound the prescribed data \(w\) on
\(\partial G\).  Perron comparison applied to
\eqref{eq:upper-comparison-equation} and
\eqref{eq:lower-comparison-equation} gives
\(w-2\eps B\leq w^h\leq w+2\eps B\), or equivalently,
\begin{equation}
\label{eq:v-barrier-bound}
 -2\eps B\leq v\leq2\eps B
 \qquad\text{in }G.
\end{equation}
Together with \eqref{eq:barrier-size}, this gives
\[
 |w-w^h|=|v|
 \leq2A\eps q_0,
\]
which proves \eqref{eq:harmonic-replacement-error} with \(C=2A\).
Taking \(w=q_0\) gives \eqref{eq:replacement-relative-error}, and
\eqref{eq:replacement-comparable} follows immediately when
\(C\eps\leq1/2\).
\end{proof}

\subsection{The quotient contraction}

\begin{proposition}
\label{prop:inhomogeneous-contraction}
Let \(A,C_0,H\geq1\) and \(L\geq0\).  After possibly reducing \(r_*\) by
an amount depending only on \(H\) and the quantitative geometric data,
there exist \(\theta,\vartheta\in(0,1)\) and \(C<\infty\), depending only
on \(A\), \(C_0\), \(H\), and the same geometric data, such that the
following holds.  Let \(P\) belong to the regular free-boundary
neighborhood and let
\(0<r\leq r_*\).  Suppose that
\[
 q_0,q_1\in C(\overline{D_{2r}(P)})
\]
are energy solutions such that \(q_0>0\) and \(q_1\geq0\) in
\(D_{2r}(P)\), and that both functions vanish continuously on the
free-boundary graph portion.  Assume that
\begin{equation}
\label{eq:inhomogeneous-system}
 \K q_i=f_i,
 \qquad
 \|f_i\|_{\mathrm L^\infty(D_{2r}(P))}\leq L,
 \qquad i=0,1.
\end{equation}
Suppose that \(q_0\) is \(H\)-balanced with respect to \(P\) at scale
\(r/2\), in the sense that
\begin{equation}
\label{eq:inhomogeneous-reference-balance}
 H^{-1}
 \leq
 \frac{q_0(A_{r/2,\Lambda}^{+}(P))}
      {q_0(A_{r/2,\Lambda}^{-}(P))}
 \leq H.
\end{equation}
Finally, assume that
\begin{equation}
\label{eq:q1-q0-comparable}
 0\leq q_1\leq Aq_0,
 \qquad
 0\leq u\leq C_0rq_0
 \qquad\text{in }D_{2r}(P).
\end{equation}
Then
\begin{equation}
\label{eq:inhomogeneous-contraction}
 \operatorname{osc}_{D_{\theta r}(P)}
 \frac{q_1}{q_0}
 \leq
 \vartheta
 \operatorname{osc}_{D_r(P)}
 \frac{q_1}{q_0}
 +
 CLr.
\end{equation}
The constants are uniform for \(P\) in the indicated regular neighborhood
and \(0<r\leq r_*\).
\end{proposition}

\begin{proof}
The reference points \(A_{r/2,\Lambda}^{\pm}(P)\) lie in the interior of
\(D_{2r}(P)\).  Since \(q_0>0\), the quotient in
\eqref{eq:inhomogeneous-reference-balance} is well defined.  No positivity
assumption on \(q_1\) at the reference points is needed.

Translate \(P\) to the origin, apply the dilation \(\delta_{1/r}\), and
normalize \(q_i\) according to the homogeneity of a first-order diffusion
derivative,
\begin{equation}
\label{eq:normalized-derivatives}
 \widetilde q_i(Z)
 =
 {r}^{-1}q_i(P\circ\delta_rZ),
 \qquad i=0,1.
\end{equation}
Then
\[
 \K\widetilde q_i(Z)
 =
 r f_i(P\circ\delta_rZ).
\]
To simplify notation, we write \(q_i\) for \(\widetilde q_i\) and \(D_s\)
for the corresponding rescaled positivity domain.  Set
\(\varepsilon=2Lr\).  The rescaled functions satisfy
\begin{equation}
\label{eq:normalized-system}
 \|\K q_i\|_{\mathrm L^\infty(D_2)}
 \leq\frac{\varepsilon}{2},
 \qquad
 u_{P,r}\leq C_0q_0,
 \qquad
 0\leq q_1\leq Aq_0
 \qquad\text{in }D_2.
\end{equation}
Moreover, \(q_0\) is \(H\)-balanced with respect to the origin at scale
\(1/2\).

Let \(\varepsilon_*>0\) be the smallness threshold in
Lemma~\ref{lem:harmonic-replacement}.  If
\(\varepsilon>\varepsilon_*\), then
\[
 0\leq\frac{q_1}{q_0}\leq A,
\]
and hence
\[
 \operatorname{osc}_{D_\theta}\frac{q_1}{q_0}
 \leq A
 \leq
 \frac{A}{\varepsilon_*}\varepsilon.
\]
Thus \eqref{eq:inhomogeneous-contraction} follows in this case after
increasing \(C\).  We may therefore assume that
\(0\leq\varepsilon\leq\varepsilon_*\).

Set \(G=D_1\), and let \(h_i\) be the \(\K\)-harmonic replacement of
\(q_i\) in \(G\),
\begin{equation}
\label{eq:hi-replacement-problem}
 \begin{cases}
  \K h_i=0 & \text{in }G,\\
  h_i=q_i  & \text{on }\partial_{\K}G,
 \end{cases}
 \qquad i=0,1.
\end{equation}
Here \(h_i=H_G(q_i|_{\partial G})\). In particular, the boundary
conditions are attained continuously on the graph portion.
Lemma~\ref{lem:harmonic-replacement}, applied with \(\rho=1\), gives
\begin{equation}
\label{eq:hi-qi-error}
 |h_i-q_i|
 \leq
 C\varepsilon q_0,
 \qquad i=0,1,
\end{equation}
and, after decreasing \(\varepsilon_*\) if necessary,
\begin{equation}
\label{eq:h0-q0-comparability}
 \frac12q_0
 \leq
 h_0
 \leq
 \frac32q_0
 \qquad\text{in }G.
\end{equation}
In particular, \(h_0>0\) in \(G\).  It also follows from
\eqref{eq:inhomogeneous-reference-balance} and
\eqref{eq:h0-q0-comparability} that \(h_0\) is \(3H\)-balanced with
respect to the origin at scale \(1/2\).

Set
\begin{equation}
\label{eq:ratio-extrema}
 \underline\lambda
 =
 \inf_G\frac{q_1}{q_0},
 \qquad
 \overline\lambda
 =
 \sup_G\frac{q_1}{q_0},
 \qquad
 \omega=\overline\lambda-\underline\lambda.
\end{equation}
By \eqref{eq:normalized-system},
\begin{equation}
\label{eq:ratio-extrema-bounds}
 0\leq\underline\lambda\leq\overline\lambda\leq A.
\end{equation}
If \(\omega+\varepsilon=0\), then \(q_1/q_0\) is constant in \(G\), and
the conclusion is immediate.  We therefore assume that
\(\omega+\varepsilon>0\).

By linearity of the Dirichlet problem,
\(h_1-\underline\lambda h_0\) is the \(\K\)-harmonic replacement of
\(q_1-\underline\lambda q_0\).  Since
\(q_1-\underline\lambda q_0\geq0\) in \(G\), its continuous boundary values are
nonnegative on \(\partial G\).  Order preservation gives
\[
 h_1-\underline\lambda h_0\geq0
 \qquad\text{in }G.
\]
Similarly,
\[
 \overline\lambda h_0-h_1\geq0
 \qquad\text{in }G.
\]
Consequently,
\begin{equation}
\label{eq:harmonic-ratio-extrema}
 \underline\lambda
 \leq
 \frac{h_1}{h_0}
 \leq
 \overline\lambda
 \qquad\text{in }G.
\end{equation}

Using the fixed parameter \(\delta\in(0,1/8)\), define the adjusted
\(\K\)-harmonic envelope
\begin{equation}
\label{eq:adjusted-envelope}
 \mathcal E
 =
 h_1-\underline\lambda h_0+(\delta\omega+\varepsilon)h_0.
\end{equation}
The function \(\mathcal E\) is strictly positive and \(\K\)-harmonic in
\(G\), and it vanishes continuously on the graph portion.  From
\eqref{eq:harmonic-ratio-extrema},
\begin{equation}
\label{eq:adjusted-envelope-bounds}
 \delta\omega+\varepsilon
 \leq
 \frac{\mathcal E}{h_0}
 \leq
 (1+\delta)\omega+\varepsilon
 \qquad\text{in }G.
\end{equation}
Since \(h_0\) is \(3H\)-balanced at scale \(1/2\), these bounds imply that
\(\mathcal E\) is \(\widehat H\)-balanced there, where
\begin{equation}
\label{eq:adjusted-balance-constant}
 \widehat H
 =
 3H{(1+\delta)}/{\delta}.
\end{equation}
Indeed,
\begin{equation}
\label{eq:adjusted-envelope-balance}
 \widehat H^{-1}
 \leq
 \frac{\mathcal E(A_{1/2,\Lambda}^{+})}
      {\mathcal E(A_{1/2,\Lambda}^{-})}
 \leq
 \widehat H.
\end{equation}
The constant \(\widehat H\) is independent of
\(\omega\), \(\varepsilon\), \(P\), and \(r\).

We now apply Theorem~\ref{thm:BHI-input} at scale \(1/2\) to
\(\mathcal E\) and \(h_0\), using the balance constant \(\widehat H\).
The radius \(r_*\) is chosen sufficiently small that
\eqref{eq:BHI-defect-hypothesis} holds with the threshold
\(\eta_0(m,M,\widehat H)\).  Let
\(\kappa=\kappa(m,M)\) and
\(\sigma=\sigma(m,M,\widehat H)\) be the constants supplied by that
theorem.  Since
\begin{equation}
\label{eq:adjusted-envelope-quotient}
 \frac{\mathcal E}{h_0}
 =
 \frac{h_1}{h_0}-\underline\lambda+\delta\omega+\varepsilon,
\end{equation}
the quotients \(\mathcal E/h_0\) and \(h_1/h_0\) have the same
oscillation.

For \(0<\theta<\kappa/2\), the quotient estimate
\eqref{eq:homogeneous-holder} and
\eqref{eq:adjusted-envelope-bounds} give
\begin{equation}
\label{eq:h-ratio-precontraction}
 \operatorname{osc}_{D_\theta}
 \frac{h_1}{h_0}
 \leq
 C\theta^\sigma(\omega+\varepsilon).
\end{equation}
Choose \(0<\theta<{\kappa(m,M)}/{2}\) sufficiently small so that
\(C(m,M,\widehat H)
 \theta^{\sigma(m,M,\widehat H)}
 \leq 1/2\).  Then
\begin{equation}
\label{eq:h-ratio-contraction}
 \operatorname{osc}_{D_\theta}
 \frac{h_1}{h_0}
 \leq
 \frac12\omega+C\varepsilon.
\end{equation}

Finally, \eqref{eq:hi-qi-error},
\eqref{eq:h0-q0-comparability}, and
\eqref{eq:harmonic-ratio-extrema} imply
\begin{equation}
\label{eq:quotient-replacement-error}
 \left|
 \frac{q_1}{q_0}
 -
 \frac{h_1}{h_0}
 \right|
 \leq
 C\varepsilon
 \qquad\text{in }G.
\end{equation}
Combining \eqref{eq:h-ratio-contraction} and
\eqref{eq:quotient-replacement-error} yields
\begin{equation}
\label{eq:normalized-inhomogeneous-contraction}
 \operatorname{osc}_{D_\theta}
 \frac{q_1}{q_0}
 \leq
 \frac12
 \operatorname{osc}_{D_1}
 \frac{q_1}{q_0}
 +
 C\varepsilon.
\end{equation}
Scaling back and recalling that \(\varepsilon=2Lr\) proves
\eqref{eq:inhomogeneous-contraction}, with
\(\vartheta=1/2\).
\end{proof}

The following elementary iteration lemma is the discrete Campanato estimate
underlying the argument, compare
\cite{AthanasopoulosCaffarelli1985}.

\begin{lemma}
\label{lem:iteration}
Let \(\omega:(0,r_*]\longrightarrow[0,\infty)\) be bounded and
non-decreasing.  Suppose that
\[
 \omega(\theta r)
 \leq
 \vartheta\omega(r)+Ar^\mu
\]
for every \(0<r\leq r_*\), where
\(\theta,\vartheta\in(0,1)\), \(A\geq0\), and \(\mu>0\).  Then, for every
\[
 0<\beta
 <
 \min\bigl\{
 \mu,
 {(\log\vartheta)}/{\log\theta}
 \bigr\},
\]
there is a constant
\(C=C(\theta,\vartheta,\mu,\beta)\) such that
\begin{equation}
\label{eq:iteration-conclusion}
 \omega(\rho)
 \leq
 C
 \left(\frac{\rho}{r}\right)^\beta
 \bigl(\omega(r)+Ar^\mu\bigr),
 \qquad
 0<\rho\leq r\leq r_*.
\end{equation}
\end{lemma}

\begin{proof}
Iterating the hypothesis at the radii \(r_k=\theta^kr\) gives
\[
 \omega(r_k)
 \leq
 \vartheta^k\omega(r)
 +
 Ar^\mu
 \sum_{j=0}^{k-1}
 \vartheta^{k-1-j}\theta^{j\mu}.
\]
The restriction on \(\beta\) implies
\[
 \vartheta\theta^{-\beta}<1,
 \qquad
 \theta^{\mu-\beta}<1.
\]
It follows that
\[
 \omega(r_k)
 \leq
 C\theta^{k\beta}
 \bigl(\omega(r)+Ar^\mu\bigr).
\]
For \(r_{k+1}<\rho\leq r_k\), monotonicity gives
\(\omega(\rho)\leq\omega(r_k)\), and
\(\theta^k\leq\theta^{-1}\rho/r\).  This proves
\eqref{eq:iteration-conclusion}.
\end{proof}

\section{H\"older continuity of the non-degenerate normal}
\label{sec:normal}

We next apply Proposition~\ref{prop:inhomogeneous-contraction} to the
positive directional derivatives introduced in
\eqref{eq:positive-combinations}.  We retain the coordinates and the
common monotonicity cone fixed in
Theorem~\ref{thm:bowman-input}.  In the positivity set, define
\begin{equation}
\label{eq:R-definition}
 R_j
 =
 \frac{u_{x_j}}{u_{x_m}},
 \qquad j=1,\ldots,m-1.
\end{equation}
Since \(u_{x_m}>0\), we have the exact identity
\begin{equation}
\label{eq:normalized-gradient-from-R}
 \frac{\nabla_xu}{|\nabla_xu|}
 =
 \frac{(R_1,\ldots,R_{m-1},1)}
 {\sqrt{1+|R|^2}}
 \qquad\text{in }\Omega_u.
\end{equation}
The corresponding boundary relation is suggested by the half-space
blow-up.  Indeed, if
\[
 U_P(x)
 =
 \frac12\bigl(x\cdot\nu(P)\bigr)_+^2,
\]
then, throughout its positivity half-space,
\[
 \frac{\partial_{x_j}U_P}{\partial_{x_m}U_P}
 =
 \frac{\nu_j(P)}{\nu_m(P)}.
\]
The common monotonicity cone gives
\(\nu_m(P)\geq c_0>0\).  Thus, once the ratios \(R_j\) are shown to have
boundary traces, those traces determine the inward unit normal through
\begin{equation}
\label{eq:normal-from-ratios}
 \nu(P)
 =
 \frac{(R_1(P),\ldots,R_{m-1}(P),1)}
 {\sqrt{1+|R(P)|^2}}.
\end{equation}

Now let \(q_0=u_{x_m}\).  Since
\begin{equation}
\label{eq:R-positive-quotients}
 {q_j^\pm}/{q_0}
 =
 1\pm a_0R_j,
\end{equation}
it is enough to establish boundary H\"older estimates for the positive
quotients \(q_j^\pm/q_0\).

\begin{proposition}
\label{prop:normal-holder}
There exist \(\beta_0\in(0,1)\), \(C<\infty\), and a fixed smaller
regular coordinate neighborhood \(\mathcal U'\) such that each \(R_j\)
extends continuously from \(\Omega_u\cap\mathcal U'\) to
\(\Gamma_u\cap\mathcal U'\).  Denoting the boundary trace by the same
symbol, we have
\begin{equation}
\label{eq:R-holder}
 |R_j(P)-R_j(\widetilde P)|
 \leq
 C(\dK(P,\widetilde P))^{\beta_0}
\end{equation}
for every \(P,\widetilde P\in\Gamma_u\cap\mathcal U'\) and
\(j=1,\ldots,m-1\).  Moreover,
\[
 R_j(P)
 =
 \frac{\nu_j(P)}{\nu_m(P)},
 \qquad
 P\in\Gamma_u\cap\mathcal U',
 \quad j=1,\ldots,m-1.
\]
Consequently, the inward unit normal is given by
\eqref{eq:normal-from-ratios} and is H\"older continuous with exponent
\(\beta_0\) on \(\Gamma_u\cap\mathcal U'\).
\end{proposition}

\begin{proof}
Fix \(j\in\{1,\ldots,m-1\}\) and one of the two signs.  We apply
Proposition~\ref{prop:inhomogeneous-contraction} with
\[
 q_1=q_j^\pm,
 \qquad
 q_0=u_{x_m}.
\]
The positivity properties established in Section~\ref{sec:derivatives}
give
\[
 q_0>0,
 \qquad
 q_j^\pm\geq0
\]
in the local positivity set.  Since both \(q_j^+\) and \(q_j^-\) are
nonnegative, \(|a_0u_{x_j}|\leq u_{x_m}\), and hence
\begin{equation}
\label{eq:qj-q0-comparison}
 0\leq q_j^\pm\leq2q_0.
\end{equation}
Lemma~\ref{lem:gradient-vanishing} shows that \(q_0\) and \(q_j^\pm\)
extend continuously by zero to the regular free-boundary graph.

Scaling back \eqref{eq:u-controlled-by-q}, with \(q=q_{0,P,r}\), gives
\begin{equation}
\label{eq:u-q0-scale-comparison}
 0\leq u\leq Crq_0
 \qquad\text{in }D_{2r}(P),
\end{equation}
after reducing the regular neighborhood and \(r_*\), if necessary.  By
\eqref{eq:derivative-equations},
\[
 \K q_0=-u_{y_m},
 \qquad
 \K q_j^\pm=-u_{y_m}\mp a_0u_{y_j}.
\]
Therefore, \eqref{eq:bowman-solution-bounds} gives
\begin{equation}
\label{eq:q-source-uniform-bound}
 \|\K q_0\|_{\mathrm L^\infty(D_{2r}(P))}
 +
 \|\K q_j^\pm\|_{\mathrm L^\infty(D_{2r}(P))}
 \leq C.
\end{equation}

Finally, apply \eqref{eq:scaled-reference-value-bounds} with center \(P\)
and scale \(r/2\).  Since
\[
 q_{0,P,r/2}(A_{1,\Lambda}^{\pm})
 =
 \frac{2}{r}\,
 q_0(A_{r/2,\Lambda}^{\pm}(P)),
\]
the common factor \(2/r\) cancels in the quotient of the two reference
values.  Hence \(q_0\) is \(H\)-balanced with respect to \(P\) at scale
\(r/2\), uniformly in \(P\) and \(r\).  Thus all the hypotheses of
Proposition~\ref{prop:inhomogeneous-contraction} are satisfied with
\[
 A=2,
 \qquad
 C_0=C,
 \qquad
 L=C.
\]

Set $F_j^\pm:={q_j^\pm}/{q_0}$. Proposition~\ref{prop:inhomogeneous-contraction} yields
\begin{equation}
\label{eq:ratio-oscillation-recursion}
 \operatorname{osc}_{D_{\theta r}(P)}F_j^\pm
 \leq
 \vartheta
 \operatorname{osc}_{D_r(P)}F_j^\pm
 +
 Cr
\end{equation}
at every graph point \(P\) and every sufficiently small \(r\).  Moreover,
\eqref{eq:qj-q0-comparison} gives
\begin{equation}
\label{eq:F-uniform-bound}
 0\leq F_j^\pm\leq2.
\end{equation}

For fixed \(P\), define
\[
 \omega_P(r)
 =
 \operatorname{osc}_{D_r(P)}F_j^\pm.
\]
This function is non-decreasing in \(r\), and
\eqref{eq:F-uniform-bound} gives the uniform initial bound
\(\omega_P(r_0)\leq2\) for a fixed sufficiently small
\(r_0\in(0,r_1]\).  Applying
Lemma~\ref{lem:iteration} to
\eqref{eq:ratio-oscillation-recursion}, with \(\mu=1\), gives an exponent
\begin{equation}
\label{eq:beta0-choice}
 0<\beta_0
 <
 \min\bigl\{
 1,{(\log\vartheta)}/{\log\theta}
 \bigr\}
\end{equation}
such that
\begin{equation}
\label{eq:ratio-boundary-decay}
 \operatorname{osc}_{D_\rho(P)}F_j^\pm
 \leq
 C\rho^{\beta_0},
 \qquad
 0<\rho\leq r_0.
\end{equation}
The constant in \eqref{eq:ratio-boundary-decay} includes the factor
\(r_0^{-\beta_0}\) from \eqref{eq:iteration-conclusion}.  Thus it retains
the dependence on the initial radius, and hence on the boundary-vanishing
modulus used to choose \(r_1\).  This is part of the fixed localization data.
The constants in the contraction estimate are independent of \(j\) and
the choice of sign.  We may therefore choose the same \(\beta_0\) for all
the functions \(F_j^\pm\).  Estimate
\eqref{eq:ratio-boundary-decay} is also uniform for \(P\) in a fixed
smaller regular neighborhood.

Because the sets \(D_\rho(P)\) form a neighborhood basis for \(P\) from
within \(\Omega_u\), and their \(F_j^\pm\)-oscillations tend to zero as
\(\rho\downarrow0\), the function \(F_j^\pm\) has a unique limit as a
point of \(\Omega_u\) approaches \(P\).  Denote this boundary value by
\(F_j^\pm(P)\).  Passing to the boundary in
\eqref{eq:R-positive-quotients} gives
\begin{equation}
\label{eq:R-boundary-value}
 R_j(P)
 =
 \frac{F_j^+(P)-1}{a_0}
 =
 \frac{1-F_j^-(P)}{a_0}.
\end{equation}
Thus \(R_j\) has a well-defined boundary trace on the regular graph.

We next compare the boundary values at two points \(P\) and
\(\widetilde P\).  Set $d=\dK(P,\widetilde P)$. Choose \(d_0>0\) sufficiently small that all boxes used below remain in
the fixed coordinate patch whenever \(d\leq d_0\).
Suppose first that \(d\leq d_0\).  The intrinsic Lipschitz geometry and
the common non-characteristic \(e_m\)-direction provide a point
\(Z\in\Omega_u\) such that
\begin{equation}
\label{eq:common-corkscrew}
 Z\in D_{C_2d}(P)\cap D_{C_2d}(\widetilde P),
 \qquad
 \dK(Z,\Gamma_u)\geq c_2d.
\end{equation}
Indeed, after translating \(P\) to the origin, one may choose a point at
height \(\lambda d\) in the \(e_m\)-direction, where \(\lambda\) depends
only on the intrinsic Lipschitz constant.  Taking \(\lambda\) sufficiently
large places the point a distance comparable to \(d\) above the graph.
Since \(\dK(P,\widetilde P)=d\), the quasi-triangle inequality then places the same
point in intrinsic boxes of radius \(C_2d\) centered at both \(P\) and
\(\widetilde P\).

Choose \(C_3\geq C_2\), and decrease \(d_0\), if necessary, so that
\(C_3d\leq r_0\) whenever \(d\leq d_0\).  Applying
\eqref{eq:ratio-boundary-decay} at \(P\) and \(\widetilde P\), at scale \(C_3d\),
and then passing to the boundary in the corresponding oscillation
estimates, we obtain
\begin{align}
 |F_j^\pm(P)-F_j^\pm(\widetilde P)|
 &\leq
 |F_j^\pm(P)-F_j^\pm(Z)|
 +
 |F_j^\pm(Z)-F_j^\pm(\widetilde P)|\leq
 Cd^{\beta_0}.
\label{eq:F-boundary-holder}
\end{align}
Equation \eqref{eq:R-boundary-value} now gives
\eqref{eq:R-holder}.  In particular, the boundary trace of \(R_j\) is
H\"older continuous.  If \(d>d_0\), the same estimate follows, after
increasing \(C\), from the uniform bound \(|R_j|\leq a_0^{-1}\), which is
a consequence of \(|a_0u_{x_j}|\leq u_{x_m}\).

It remains to identify this trace with the inward normal.  Let
\(\nu_{\mathrm{bu}}(P)\) denote the inward normal of the unique half-space
blow-up at \(P\).  Its \(m\)-th component is uniformly bounded away from
zero, so the central reference point
\[
 Z_*=A_{1,\Lambda}
\]
lies strictly inside the limiting positivity half-space.  The convergence
\begin{equation}
\label{eq:blowup-C1-convergence}
 u_{P,r}
 \longrightarrow
 \frac12\bigl(x\cdot\nu_{\mathrm{bu}}(P)\bigr)_+^2
\end{equation}
in \(C_x^1\) on compact subsets of that half-space gives
\begin{equation}
\label{eq:R-normal-identification}
 \lim_{r\downarrow0}
 \frac{u_{x_j}(P\circ\delta_rZ_*)}
      {u_{x_m}(P\circ\delta_rZ_*)}
 =
 \frac{\nu_{\mathrm{bu},j}(P)}
      {\nu_{\mathrm{bu},m}(P)}.
\end{equation}
Here the common scaling factor \(r^{-1}\) in the diffusion derivatives
cancels in the quotient, and the limiting denominator is strictly
positive.  The points \(P\circ\delta_rZ_*\) belong to \(\Omega_u\) for all
sufficiently small \(r\) and converge intrinsically to \(P\).  Hence the
left-hand side of \eqref{eq:R-normal-identification} also converges to the
boundary trace \(R_j(P)\).  Consequently,
\begin{equation}
\label{eq:R-blowup-normal}
 R_j(P)
 =
 \frac{\nu_{\mathrm{bu},j}(P)}
      {\nu_{\mathrm{bu},m}(P)},
 \qquad
 j=1,\ldots,m-1.
\end{equation}
Since \(\nu_{\mathrm{bu},m}(P)>0\), this gives
\[
 \nu_{\mathrm{bu}}(P)
 =
 \frac{(R(P),1)}{\sqrt{1+|R(P)|^2}}.
\]

Finally, let \(a_P\) denote the intrinsic first-order differential of the
graph at \(P\).  In intrinsic coordinates centered at \(P\), the
qualitative Kolmogorov differentiability of the graph gives convergence
of its rescalings to the tangent hyperplane $x_m=a_P\cdot x'$. On the other hand, the rescaled positivity sets converge to the positivity
half-space of the unique blow-up,
\[
 \{x\cdot\nu_{\mathrm{bu}}(P)>0\},
\]
whose boundary is
\[
 x_m
 =
 -\frac{\nu_{\mathrm{bu}}'(P)}
        {\nu_{\mathrm{bu},m}(P)}
 \cdot x'
 =
 -R(P)\cdot x'.
\]
Because these are the limits of the same rescaled free-boundary graphs,
the two limiting hyperplanes coincide.  Therefore,
\begin{equation}
\label{eq:graph-gradient-ratio}
 \nabla_{x'}\psi(P)
 =
 a_P
 =
 -R(P).
\end{equation}
Thus the inward blow-up normal agrees with the intrinsic graph normal.
In view of \eqref{eq:R-blowup-normal}, this common normal is precisely the
vector in \eqref{eq:normal-from-ratios}.  Since \(R\) is H\"older
continuous and the map
\[
 R\longmapsto\frac{(R,1)}{\sqrt{1+|R|^2}}
\]
is smooth on the bounded range of \(R\), the inward unit normal is
H\"older continuous with exponent \(\beta_0\).  This completes the proof.
\end{proof}

\section{Power decay of the transport derivative}
\label{sec:transport}

We continue to use the commutator convention
\([A,B]=AB-BA\).  Since $\K=\Delta_x+Y$, $Y=x\cdot\nabla_y-\partial_t$, a direct computation gives
\begin{equation}
\label{eq:K-Y-commutator}
 [\K,Y]
 =
 [\Delta_x,Y]
 =
 2\sum_{i=1}^m\partial_{x_i}\partial_{y_i}.
\end{equation}
Because \(\K u=1\) in \(\Omega_u\) and \(Y1=0\), it follows that
\begin{equation}
\label{eq:Y-equation}
 \K(Yu)
 =
 2\sum_{i=1}^m\partial_{x_i}\partial_{y_i}u
 =
 2\operatorname{div}_x(\nabla_yu)
 \qquad\text{in }\Omega_u.
\end{equation}
We use \eqref{eq:Y-equation} in the distributional sense.  In this form,
the right-hand side involves only the bounded vector field
\(\nabla_yu\) and no pointwise control of the mixed derivatives
\(\partial_{x_i}\partial_{y_i}u\) is required.

The transport derivative is invariant under the obstacle-problem
rescaling, whereas a \(y\)-derivative gains one power of the scale.
More precisely,
\begin{equation}
\label{eq:scaled-Y-identities}
 Yu_{P,r}(Z)
 =
 Yu(P\circ\delta_rZ),
 \qquad
 \nabla_yu_{P,r}(Z)
 =
 r\nabla_yu(P\circ\delta_rZ).
\end{equation}
Consequently, \eqref{eq:Y-equation} becomes
\begin{equation}
\label{eq:scaled-Y-equation}
 \K(Yu_{P,r})
 =
 \operatorname{div}_x\mathbf F_{P,r},
 \qquad
 \mathbf F_{P,r}
 =
 2\nabla_yu_{P,r},
\end{equation}
and \eqref{eq:bowman-solution-bounds} gives
\begin{equation}
\label{eq:scaled-Y-vector-bound}
 \|\mathbf F_{P,r}\|_{\mathrm L^\infty}
 \leq Cr
\end{equation}
on every fixed rescaled working box.  Moreover,
\cite[Proposition~6.7]{Bowman2025} shows that \(Yu\) extends continuously
to the regular free boundary with boundary value zero.  We next quantify
this vanishing.

\begin{lemma}
\label{lem:divergence-perturbation}
Let \(\Omega_\psi\) be one of the normalized graph domains arising from the
regular free-boundary neighborhood, with \(0\in\Delta_\psi\), and set
\(D_s=\Omega_\psi\cap Q_{M,s}(0)\).  Let \(g\) be a bounded
distributional solution of
\begin{equation}
\label{eq:divergence-data-equation}
 \K g=\operatorname{div}_xF
 \qquad\text{in }D_2,
\end{equation}
where \(F\in\mathrm L^\infty(D_2;\R^m)\).  Suppose that \(g\) extends
continuously to the graph portion of \(\partial D_2\) and vanishes there.
Then there exist
\[
 \theta_0\in(0,1/4),
 \qquad
 \alpha_0\in(0,1),
 \qquad
 C<\infty,
\]
such that
\begin{equation}
\label{eq:divergence-boundary-decay}
 \|g\|_{\mathrm L^\infty(D_{2\theta})}
 \leq
 C\theta^{\alpha_0}
 \|g\|_{\mathrm L^\infty(D_2)}
 +
 C\|F\|_{\mathrm L^\infty(D_2)}
\end{equation}
for every \(0<\theta\leq\theta_0\).  The constants depend only on the
quantitative geometric data and are uniform for all normalized graph
domains arising in the regular neighborhood.
\end{lemma}

\begin{proof}
Choose \(\chi\in C_0^\infty(Q_{M,2}(0))\) such that
\[
 0\leq\chi\leq1,
 \qquad
 \chi\equiv1
 \quad\text{on }Q_{M,3/2}(0).
\]
Define the compactly supported vector field
\(\widetilde F\) on the ambient Kolmogorov group
\((\R^{2m+1},\circ)\) by
\begin{equation}
\label{eq:F-extension}
 \widetilde F(Z)
 =
 \begin{cases}
  \chi(Z)F(Z), & Z\in D_2,\\
  0,           & Z\notin D_2.
 \end{cases}
\end{equation}
Let \(\Gamma_{\K}\) be the fundamental solution normalized by
\(\K\Gamma_{\K}=\delta_0\), and set
\(E(Z,W)=\Gamma_{\K}(W^{-1}\circ Z)\).  Left invariance gives
\(\K_ZE(Z,W)=\delta_W\).  Define
\begin{equation}
\label{eq:divergence-potential}
 z(Z)
 =
 -\sum_{i=1}^m
 \int_{\R^{2m+1}}
 \partial_{w_i}E(Z,W)\widetilde F_i(W)\,\d W.
\end{equation}
This is the potential of \(\operatorname{div}_x\widetilde F\), with the
order of the two kernel variables fixed explicitly.  The derivatives
\(\partial_{w_i}\) act on the diffusion coordinates of \(W\).
Writing \(K_i(Z,W)=-\partial_{w_i}E(Z,W)\), these kernels are
first-order right-invariant derivatives of \(\Gamma_{\K}\).
If \(\mathfrak q=4m+2\) denotes the homogeneous
dimension, then \(\Gamma_{\K}\) has degree
\(2-\mathfrak q\), while \(K_i\) has degree \(1-\mathfrak q\).  In
particular,
\[
 |K_i(Z,W)|
 \leq
 C\|W^{-1}\circ Z\|_{\K}^{1-\mathfrak q}.
\]
These kernels are locally integrable with respect to Haar measure, which in
the present coordinates is Lebesgue measure.  The standard fundamental
solution and kernel estimates may be found, for example, in
\cite{NystromPolidoro2016}.

Since \(\widetilde F\) is supported in a fixed box, the kernel estimate
gives
\begin{equation}
\label{eq:potential-Linfty-bound}
 \|z\|_{\mathrm L^\infty(Q_{M,3/2}(0))}
 \leq
 C\|F\|_{\mathrm L^\infty(D_2)}.
\end{equation}
Translation continuity of the kernels in
\(\mathrm L^1_{\mathrm{loc}}\) also shows that \(z\) has a continuous
representative.  Since \(\widetilde F=F\) in \(D_{3/2}\), we have
\begin{equation}
\label{eq:potential-equation}
 \K z
 =
 \operatorname{div}_xF
 \qquad\text{in }D_{3/2}
\end{equation}
in the sense of distributions.  The extension of \(\widetilde F\) across
the graph produces no additional term in
\eqref{eq:potential-equation}, since the identity is asserted only in the
open set \(D_{3/2}\).

Set $h=g-z$. Equations \eqref{eq:divergence-data-equation} and
\eqref{eq:potential-equation} imply that
\[
 \K h=0
 \qquad\text{in }D_{3/2}.
\]
On the graph portion, \(g=0\), and hence \(h=-z\).  We remove this
nonzero boundary value by solving an auxiliary homogeneous Dirichlet
problem.

Fix \(\rho\in(5/4,3/2)\), set \(G=D_\rho\), and choose
\(\eta\in C_0^\infty(Q_{M,\rho}(0))\) such that
\begin{equation}
\label{eq:boundary-cutoff}
 0\leq\eta\leq1,
 \qquad
 \eta\equiv1
 \quad\text{on }Q_{M,1}(0).
\end{equation}
Let \(b=H_G((-\eta z)|_{\partial G})\) be the Perron solution of
\begin{equation}
\label{eq:boundary-correction-problem}
 \begin{cases}
  \K b=0 & \text{in }G,\\
  b=-\eta z & \text{on }\partial_{\K}G.
 \end{cases}
\end{equation}
The boundary data are continuous and vanish near the artificial boundary.
Perron comparison and
\eqref{eq:potential-Linfty-bound} give
\begin{equation}
\label{eq:boundary-correction-bound}
 \|b\|_{\mathrm L^\infty(G)}
 \leq
 \|\eta z\|_{\mathrm L^\infty(\partial G)}
 \leq
 C\|F\|_{\mathrm L^\infty(D_2)}.
\end{equation}

Define
\begin{equation}
\label{eq:zero-boundary-harmonic-remainder}
 \widehat h
 =
 h-b
 =
 g-z-b.
\end{equation}
Then \(\widehat h\) is \(\K\)-harmonic in \(G\).  On the graph portion
contained in \(Q_{M,1}(0)\), we have \(g=0\), \(h=-z\), and \(b=-z\).
Therefore,
\begin{equation}
\label{eq:hhat-zero-boundary}
 \widehat h=0
 \qquad\text{on the graph portion in }Q_{M,1}(0).
\end{equation}
The local boundary H\"older estimate for bounded, possibly sign-changing
\(\K\)-harmonic functions vanishing on a non-characteristic graph portion
now gives
\begin{equation}
\label{eq:hhat-boundary-holder}
 \|\widehat h\|_{\mathrm L^\infty(D_{2\theta})}
 \leq
 C\theta^{\alpha_0}
 \|\widehat h\|_{\mathrm L^\infty(D_1)},
 \qquad
 0<\theta\leq\theta_0.
\end{equation}
This estimate requires only the intrinsic Lipschitz character of the graph,
not independence from \(y_m\), see
\cite[Theorem~3.2]{LitsgardNystrom2022}.

By \eqref{eq:potential-Linfty-bound} and
\eqref{eq:boundary-correction-bound},
\begin{equation}
\label{eq:hhat-global-bound}
 \|\widehat h\|_{\mathrm L^\infty(D_1)}
 \leq
 \|g\|_{\mathrm L^\infty(D_2)}
 +
 C\|F\|_{\mathrm L^\infty(D_2)}.
\end{equation}
Since \(g=\widehat h+z+b\), estimates
\eqref{eq:potential-Linfty-bound},
\eqref{eq:boundary-correction-bound},
\eqref{eq:hhat-boundary-holder}, and
\eqref{eq:hhat-global-bound} give
\begin{align}
 \|g\|_{\mathrm L^\infty(D_{2\theta})}
 &\leq
 \|\widehat h\|_{\mathrm L^\infty(D_{2\theta})}
 +
 \|z\|_{\mathrm L^\infty(D_{2\theta})}
 +
 \|b\|_{\mathrm L^\infty(D_{2\theta})}\leq
 C\theta^{\alpha_0}
 \|g\|_{\mathrm L^\infty(D_2)}
 +
 C\|F\|_{\mathrm L^\infty(D_2)}.
\label{eq:divergence-boundary-decay-proof}
\end{align}
This proves \eqref{eq:divergence-boundary-decay}.
\end{proof}

\begin{proposition}
\label{prop:Y-power-decay}
There exist \(\gamma\in(0,1)\) and \(C<\infty\) such that
\begin{equation}
\label{eq:Y-power-decay}
 \sup_{\Omega_u\cap Q_{M,r}(P)}|Yu|
 \leq
 Cr^\gamma
\end{equation}
for every regular free-boundary point \(P\) in a fixed smaller neighborhood
and every \(0<r\leq r_*\).
\end{proposition}

\begin{proof}
For a regular free-boundary point \(P\), define
\begin{equation}
\label{eq:Y-scale-supremum}
 a_P(r)
 =
 \sup_{\Omega_u\cap Q_{M,r}(P)}|Yu|.
\end{equation}
The function \(a_P\) is non-decreasing.  After reducing \(r_*\), if
necessary, we may assume that all boxes at twice the working scale remain
in the regular coordinate neighborhood.

Fix \(0<r\leq r_*\), translate \(P\) to the origin, and apply the dilation
\(\delta_{1/r}\).  By \eqref{eq:scaled-Y-identities},
\begin{equation}
\label{eq:Y-supremum-rescaling}
 \|Yu_{P,r}\|_{\mathrm L^\infty(D_2)}
 =
 a_P(2r).
\end{equation}
Moreover, \eqref{eq:scaled-Y-vector-bound} gives
\begin{equation}
\label{eq:scaled-Y-source-bound}
 \|\mathbf F_{P,r}\|_{\mathrm L^\infty(D_2)}
 \leq
 Cr.
\end{equation}
Since \(Yu_{P,r}\) vanishes continuously on the rescaled graph, we may
apply Lemma~\ref{lem:divergence-perturbation} with $g=Yu_{P,r}$, $F=\mathbf F_{P,r}$. We obtain
\begin{equation}
\label{eq:Y-recursion-r}
 a_P(2\theta r)
 \leq
 C_1\theta^{\alpha_0}a_P(2r)
 +
 Cr
\end{equation}
for every \(0<\theta\leq\theta_0\).  Writing \(R=2r\) and absorbing the
factor \(1/2\) into \(C\), this becomes
\begin{equation}
\label{eq:Y-recursion-R}
 a_P(\theta R)
 \leq
 C_1\theta^{\alpha_0}a_P(R)
 +
 CR.
\end{equation}

Choose \(\theta\in(0,\theta_0]\) sufficiently small that
\begin{equation}
\label{eq:Y-contraction-factor}
 \vartheta_Y
 :=
 C_1\theta^{\alpha_0}
 <1.
\end{equation}
Then
\begin{equation}
\label{eq:Y-fixed-scale-recursion}
 a_P(\theta R)
 \leq
 \vartheta_Y a_P(R)
 +
 CR.
\end{equation}
Choose
\begin{equation}
\label{eq:Y-decay-exponent}
 0<\gamma
 <
 \min\bigl\{
 1,{(\log\vartheta_Y)}/{\log\theta}
 \bigr\}.
\end{equation}
Both logarithms are negative, so their ratio is positive.
Lemma~\ref{lem:iteration}, applied to
\eqref{eq:Y-fixed-scale-recursion}, yields
\begin{equation}
\label{eq:Y-iteration-conclusion}
 a_P(r)
 \leq
 Cr^\gamma,
 \qquad
 0<r\leq r_*.
\end{equation}
The initial-scale bound is uniform because $Yu=x\cdot\nabla_yu-\partial_tu$ is locally bounded by \eqref{eq:bowman-solution-bounds}.  All geometric and
analytic constants are uniform for \(P\) in the smaller regular
neighborhood.  This proves \eqref{eq:Y-power-decay}.
\end{proof}

\section{Improvement of flatness and proof of the main theorem}
\label{sec:flatness}

We now convert the analytic estimates obtained above into geometric
estimates for the free-boundary graph. For \(P=(x_P,y_P,t_P)\), the integral curve of \(Y\) through \(P\) is given by
\begin{equation}
\label{eq:Y-flow}
 \Phi_s(P)
 =
 P_s
 =
 (x_P,y_P+s x_P,t_P-s).
\end{equation}
Even if \(P\in\Gamma_u\), the point \(P_s\) need not belong to the free
boundary.  For any point \(Z=(x_Z,y_Z,t_Z)\), we write
\(x_m(Z):=(x_Z)_m\) for the \(m\)-th component of its diffusion coordinate.

\begin{lemma}
\label{lem:flow-flatness}
Let \(P\in\Gamma_u\) belong to a fixed smaller regular coordinate patch,
let \(0<r\leq r_*\), and suppose that \(|s|\leq r^2\).  Let
\(\widetilde P_s\) be the unique free-boundary graph point having the same
\((x',y,t)\)-coordinates as \(P_s=\Phi_s(P)\).  Then
\begin{equation}
\label{eq:flow-flatness}
 |h_s|
 \leq
 Cr^{1+\gamma/2},\qquad  h_s
 :=
 x_m(P_s)-x_m(\widetilde P_s).
\end{equation}
The constant is uniform for \(P\) in the indicated patch.
\end{lemma}

\begin{proof} The points \(P\) and \(\widetilde P_s\) lie on the free-boundary graph, while the
base coordinates of \(\widetilde P_s\) agree with those of
\(P_s=\Phi_s(P)\).  Since the \(Y\)-flow leaves the diffusion coordinate
unchanged, \(x_m(P_s)=x_m(P)\).  Moreover, the transport correction in
the boundary quasi-distance cancels the displacement
\(y(P_s)-y(P)=sx_P\).  Hence
\[
 D_\psi(\zeta_{\widetilde P_s},\zeta_P)
 \leq
 C|s|^{1/2}
 \leq Cr.
\]
The intrinsic Lipschitz character of the regular free-boundary graph,
supplied by
Theorem~\ref{thm:bowman-input}\textnormal{(ii)}, therefore gives
\begin{equation}
\label{eq:flow-height-preliminary}
 |h_s|
 =
 |x_m(P)-x_m(\widetilde P_s)|
 \leq Cr.
\end{equation}
After reducing \(r_*\), if necessary, we may assume that all points and
characteristic segments introduced below remain in a fixed enlargement
\(Q_{M,Cr}(P)\) of the working box.

The bounds for \(\partial_tu\) and \(\nabla_yu\) imply that \(u\) is
absolutely continuous along \(Y\)-characteristics and that
\begin{equation}
\label{eq:characteristic-fundamental-theorem}
 u(\Phi_s(Z))-u(Z)
 =
 \int_0^s Yu(\Phi_\tau(Z))\,\mathrm d\tau
\end{equation}
whenever the characteristic segment remains in the coordinate patch.  We
extend \(Yu\) by zero on the contact set.  This agrees almost everywhere
with the derivative of \(u\) along each characteristic, since an
absolutely continuous nonnegative function has derivative zero almost
everywhere on its zero set.  The extension is also compatible with the
continuous vanishing of \(Yu\) on the regular free boundary.

Since \(u(P)=0\), Proposition~\ref{prop:Y-power-decay} and
\eqref{eq:characteristic-fundamental-theorem} give
\begin{equation}
\label{eq:u-along-characteristic}
 |u(P_s)|
 \leq
 |s|
 \sup_{\Omega_u\cap Q_{M,Cr}(P)}|Yu|
 \leq
 Cr^{2+\gamma}.
\end{equation}

Suppose first that \(h_s\geq0\).  Since \(P_s\) and \(\widetilde P_s\) have the
same \((x',y,t)\)-coordinates, $P_s=\widetilde P_s+h_se_m$. Thus \(P_s\) lies above \(\widetilde P_s\) in the \(e_m\)-direction.  This direction
belongs to the fixed non-degeneracy cone, and
\eqref{eq:flow-height-preliminary} gives \(0\leq h_s\leq Cr\).
The scaled form of the quadratic cone lower bound
\eqref{eq:quadratic-cone-lower}, applied at a scale comparable to \(r\)
and centered at \(\widetilde P_s\), therefore gives
\begin{equation}
\label{eq:positive-flow-height}
 ch_s^2
 \leq
 u(\widetilde P_s+h_se_m)
 =
 u(P_s).
\end{equation}
Combining \eqref{eq:positive-flow-height} with
\eqref{eq:u-along-characteristic}, we obtain $h_s\leq Cr^{1+\gamma/2}$.

Suppose next that \(h_s<0\).  Flow \(\widetilde P_s\) for the time \(-s\)
and set
\begin{equation}
\label{eq:reverse-flow-point}
 P^\sharp
 =
 \Phi_{-s}(\widetilde P_s).
\end{equation}
Since $x(\widetilde P_s)=x(P)-h_se_m$, a direct computation using \eqref{eq:Y-flow} gives
\begin{equation}
\label{eq:reverse-flow-error}
 x_m(P^\sharp)-x_m(P)=-h_s,
 \qquad
 y(P^\sharp)-y(P)=s h_se_m,
 \qquad
 t(P^\sharp)=t(P).
\end{equation}
Let \(\widehat P\in\Gamma_u\) be the unique graph point with base
coordinates $(x'_P,y_P+s h_se_m,t_P)$. By \eqref{eq:flow-height-preliminary}, $|s h_s|\leq Cr^3$, so \(\widehat P\) remains in the same coordinate patch.  Since the base
coordinates of \(P\) and \(\widehat P\) differ only in the
\(y_m\)-variable, \eqref{eq:ym-half-holder} gives
\begin{equation}
\label{eq:reverse-flow-y-error}
 |x_m(\widehat P)-x_m(P)|
 \leq
 C|s h_s|^{1/2}
 \leq
 Cr|h_s|^{1/2}.
\end{equation}

Set
\begin{equation}
\label{eq:reverse-flow-height}
 k
 =
 x_m(P^\sharp)-x_m(\widehat P).
\end{equation}
The points \(P^\sharp\) and \(\widehat P\) have the same
\((x',y,t)\)-coordinates.  Moreover,
\eqref{eq:reverse-flow-error} gives
\begin{equation}
\label{eq:k-h-relation}
 k
 =
 -h_s-\bigl(x_m(\widehat P)-x_m(P)\bigr).
\end{equation}

If \(k\leq0\), then \eqref{eq:k-h-relation} and
\(-h_s=|h_s|\) imply
\[
 |h_s|
 \leq
 x_m(\widehat P)-x_m(P)
 \leq
 \bigl|x_m(\widehat P)-x_m(P)\bigr|.
\]
Using \eqref{eq:reverse-flow-y-error}, we obtain $|h_s|
 \leq
 Cr|h_s|^{1/2}$. Since \(h_s\neq0\) in the case under consideration, division by
\(|h_s|^{1/2}\) and squaring give
\begin{equation}
\label{eq:negative-height-contact-case}
 |h_s|\leq Cr^2.
\end{equation}

It remains to consider \(k>0\).  In this case $P^\sharp=\widehat P+ke_m$ lies above \(\widehat P\) in the \(e_m\)-direction and therefore belongs
to \(\Omega_u\).  Since \(\widetilde P_s\in\Gamma_u\), integration along the
characteristic joining \(\widetilde P_s\) to
\(P^\sharp=\Phi_{-s}(\widetilde P_s)\) gives
\begin{equation}
\label{eq:u-reverse-characteristic}
 |u(P^\sharp)|
 \leq
 |s|
 \sup_{\Omega_u\cap Q_{M,Cr}(P)}|Yu|
 \leq
 Cr^{2+\gamma}.
\end{equation}
Equations \eqref{eq:flow-height-preliminary} and
\eqref{eq:reverse-flow-y-error} also give \(0<k\leq Cr\).  The
displacement from \(\widehat P\) to \(P^\sharp\) is in the
\(e_m\)-direction, which lies in the fixed non-degeneracy cone.
Consequently, the scaled quadratic cone lower bound
\eqref{eq:quadratic-cone-lower}, applied at \(\widehat P\) at a scale
comparable to \(r\), gives
\begin{equation}
\label{eq:k-quadratic-lower}
 ck^2
 \leq
 u(P^\sharp).
\end{equation}
Combining \eqref{eq:u-reverse-characteristic} and
\eqref{eq:k-quadratic-lower}, we obtain
\begin{equation}
\label{eq:k-upper}
 k
 \leq
 Cr^{1+\gamma/2}.
\end{equation}

Since \(h_s<0\), equation \eqref{eq:k-h-relation} can be written as $|h_s|
 =
 k+x_m(\widehat P)-x_m(P)$. Therefore, \eqref{eq:reverse-flow-y-error} and
\eqref{eq:k-upper} yield
\begin{equation}
\label{eq:h-pre-absorption}
 |h_s|
 \leq
 Cr^{1+\gamma/2}
 +
 Cr|h_s|^{1/2}.
\end{equation}
Hence
\[
 Cr|h_s|^{1/2}
 \leq
 |h_s|/2+Cr^2,
\]
and after absorption,
\[
 |h_s|
 \leq
 Cr^{1+\gamma/2}+Cr^2.
\]
Since \(0<\gamma<1\) and \(r\leq1\), $r^2\leq r^{1+\gamma/2}$. Thus $|h_s|\leq Cr^{1+\gamma/2}$ also in the case \(k>0\).  The same conclusion follows from
\eqref{eq:negative-height-contact-case} when \(k\leq0\), and the
\(h_s\geq0\) case was established above.  This proves
\eqref{eq:flow-flatness}.
\end{proof}

Lemma~\ref{lem:flow-flatness} controls the variation of the graph in the
transport direction \(Y\), that is, under the coupled displacement $(y,t)\longmapsto(y+s x,t-s)$. This is the intrinsic analogue of variation in time and has homogeneous
degree two.  We also record the improvement for pure displacements in the
degree-three variables.  If two free-boundary points have the same
\((x',t)\)-coordinates and their \(y\)-coordinates differ by at most
\(Cr^3\), then \eqref{eq:bowman-graph-regularity} gives
\begin{equation}
\label{eq:y-flatness}
 |\Delta x_m|
 \leq
 C|\Delta y|^{1/2}
 \leq
 Cr^{3/2}.
\end{equation}
Thus the graph variation under a degree-three displacement gains a factor
\(r^{1/2}\) relative to the first-order scale.

We next state the geometric principle that combines the estimates in the
three relevant directions.  Variation in the tangential diffusion
variables \(x'\) is controlled by the H\"older continuity of the
non-degenerate normal, variation in the coupled transport direction \(Y\)
is controlled by Lemma~\ref{lem:flow-flatness}, and the remaining pure
variation in the degree-three variables \(y\) is controlled by
\eqref{eq:y-flatness}.  Together, these estimates yield a quantitative
first-order approximation of the free-boundary graph in the full
Kolmogorov geometry.

\begin{lemma}
\label{lem:geometric-upgrade}
Let \(S=\Delta_\psi\) be a non-characteristic intrinsic Lipschitz graph
in a fixed coordinate patch, and suppose that \(S\) is differentiable in
the Kolmogorov sense at every point.  Assume that, for some
\(\beta_0\in(0,1)\), \(c_0>0\), and \(C_*\geq1\), its inward
non-degenerate unit normal satisfies
\begin{equation}
\label{eq:geometric-normal-holder}
 |\nu(P)-\nu(\widetilde P)|
 \leq
 C_*(\dK(P,\widetilde P))^{\beta_0},
 \qquad
 P,\widetilde P\in S,
\end{equation}
and
\begin{equation}
\label{eq:geometric-normal-transversality}
 \nu_m(P)\geq c_0>0,
 \qquad
 P\in S.
\end{equation}
Suppose, in addition, that there exist
\(\beta_y,\beta_t\in(0,1)\) and \(r_0>0\) such that the following
estimates hold uniformly for every \(0<r\leq r_0\):

\begin{enumerate}[label=\textnormal{(\roman*)}]
\item if \(P,\widetilde P\in S\) have the same \((x',t)\)-coordinates and
\[
 |y_{\widetilde P}-y_P|\leq C_*r^3,
\]
then
\begin{equation}
\label{eq:geometric-y-hypothesis}
 |x_m(\widetilde P)-x_m(P)|
 \leq
 C_*r^{1+\beta_y};
\end{equation}

\item if \(P\in S\), \(|s|\leq C_*r^2\), and \(\widetilde P_s\in S\) is the unique
graph point having the same base coordinates as \(\Phi_s(P)\), then
\begin{equation}
\label{eq:geometric-flow-hypothesis}
 |x_m(\Phi_s(P))-x_m(\widetilde P_s)|
 \leq
 C_*r^{1+\beta_t}.
\end{equation}
\end{enumerate}
Then, with
\begin{equation}
\label{eq:beta-choice}
 \beta
 =
 \min\{\beta_0,\beta_y,\beta_t\},
\end{equation}
the graph is locally \(C_{\K}^{1,\beta}\).  More precisely, in every
smaller coordinate patch there exists \(C<\infty\) such that
\begin{equation}
\label{eq:geometric-normal-flatness}
 |\nu(P)\cdot(x_{\widetilde P}-x_P)|
 \leq
 C(\dK(P,\widetilde P))^{1+\beta}
\end{equation}
for all \(P,\widetilde P\in S\) in that patch.
\end{lemma}

\begin{proof}
Fix a smaller coordinate patch whose closure is contained in the original
one, and let \(P,\widetilde P\in S\) belong to this smaller patch.  Set
\(r=\dK(P,\widetilde P)\).  If \(r\) is bounded below by a fixed positive
fraction of the patch radius, then \eqref{eq:geometric-normal-flatness}
follows, after increasing
\(C\), from the boundedness of the patch.  We may therefore assume that
\(r\) is sufficiently small that all points constructed below remain in
the original coordinate patch and that the hypotheses of the lemma apply.

Since the graph is differentiable in the Kolmogorov sense and
\(\nu_m\geq c_0>0\), its differential in the tangential diffusion
variables is
\begin{equation}
\label{eq:graph-slope-normal}
 \nabla_{x'}\psi(P)
 =
 -{\nu'(P)}/{\nu_m(P)}.
\end{equation}
The map \(\nu\longmapsto-{\nu'}/{\nu_m}\) is Lipschitz on the region
\(\{\nu_m\geq c_0\}\).  Hence \eqref{eq:geometric-normal-holder} implies
\begin{equation}
\label{eq:graph-slope-holder}
 |\nabla_{x'}\psi(P)-\nabla_{x'}\psi(\widetilde P)|
 \leq
 C(\dK(P,\widetilde P))^{\beta_0}.
\end{equation}

We connect \(P\) to \(\widetilde P\) through two intermediate graph points.  First,
change \(x'\) from \(x'_P\) to \(x'_{\widetilde P}\), keeping
\((y,t)=(y_P,t_P)\) fixed, and set
\[
 A
 =
 \bigl(x'_{\widetilde P},\psi(x'_{\widetilde P},y_P,t_P),y_P,t_P\bigr)
 \in S.
\]
Every graph point on this \(x'\)-segment lies at intrinsic distance at
most \(Cr\) from \(P\).  Integrating
\eqref{eq:graph-slope-normal} along the segment and using
\eqref{eq:graph-slope-holder}, we obtain
\begin{align}
 &\left|
 x_m(A)-x_m(P)
 +
 \frac{\nu'(P)}{\nu_m(P)}
 \cdot(x'_{\widetilde P}-x'_P)
 \right|\notag\\*
 &\quad\leq
 \int_0^1
 \left|
 \nabla_{x'}\psi\bigl(x'_P+\tau(x'_{\widetilde P}-x'_P),y_P,t_P\bigr)
 -
 \nabla_{x'}\psi(P)
 \right|
 |x'_{\widetilde P}-x'_P|\,\mathrm d\tau\leq
 Cr^{1+\beta_0}.
\label{eq:xprime-flatness}
\end{align}

Next, set $s=t_P-t_{\widetilde P}$. Since \(r=\dK(P,\widetilde P)\), we have \(|s|\leq Cr^2\).  Fix a
geometric constant \(\lambda\geq1\), large enough that
\(|s|\leq C_*(\lambda r)^2\), and reduce the admissible upper bound for
\(r\) so that \(\lambda r\leq r_0\).  Let \(B\in S\) be the unique graph
point having the same base coordinates as \(\Phi_s(A)\).
The \(Y\)-flow leaves the diffusion coordinate unchanged, so
\[
 x_m(\Phi_s(A))=x_m(A).
\]
Hypothesis \textnormal{(ii)}, applied at scale \(\lambda r\), therefore
gives
\begin{equation}
\label{eq:A-B-flow-error}
 |x_m(B)-x_m(A)|
 \leq
 Cr^{1+\beta_t}.
\end{equation}
Moreover,
\begin{equation}
\label{eq:B-coordinates}
 x'_B=x'_{\widetilde P},
 \qquad
 t_B=t_{\widetilde P},
 \qquad
 y_B=y_P+s\,x(A).
\end{equation}

Thus \(B\) and \(\widetilde P\) have the same \((x',t)\)-coordinates.  By the
definition of the symmetrized quasi-distance,
\begin{equation}
\label{eq:P-Q-transport-error}
 |y_{\widetilde P}-y_P+(t_{\widetilde P}-t_P)x(P)|
 \leq
 Cr^3.
\end{equation}
Since \(s=t_P-t_{\widetilde P}\), equation \eqref{eq:B-coordinates} gives
\begin{align*}
 y_{\widetilde P}-y_B
 &=
 y_{\widetilde P}-y_P+(t_{\widetilde P}-t_P)x(A)=
 y_{\widetilde P}-y_P+(t_{\widetilde P}-t_P)x(P)
 +(t_{\widetilde P}-t_P)\bigl(x(A)-x(P)\bigr).
\end{align*}
The intrinsic Lipschitz character of the graph and
\(|x'_A-x'_P|\leq r\) imply $|x(A)-x(P)|\leq Cr$. Combining this estimate with
\eqref{eq:P-Q-transport-error} and
\(|t_{\widetilde P}-t_P|\leq Cr^2\), we find
\begin{equation}
\label{eq:B-Q-y-error}
 |y_{\widetilde P}-y_B|
 \leq
 Cr^3.
\end{equation}
After increasing \(\lambda\), if necessary, and reducing the admissible
upper bound for \(r\) again, we have
\(|y_{\widetilde P}-y_B|\leq C_*(\lambda r)^3\).  Hypothesis
\textnormal{(i)}, applied at scale \(\lambda r\) to \(B\) and
\(\widetilde P\), now gives
\begin{equation}
\label{eq:B-Q-height-error}
 |x_m(\widetilde P)-x_m(B)|
 \leq
 Cr^{1+\beta_y}.
\end{equation}

Adding \eqref{eq:xprime-flatness},
\eqref{eq:A-B-flow-error}, and
\eqref{eq:B-Q-height-error}, and recalling
\eqref{eq:beta-choice}, we obtain
\begin{equation}
\label{eq:graph-first-order-remainder}
 \left|
 x_m(\widetilde P)-x_m(P)
 +
 \frac{\nu'(P)}{\nu_m(P)}
 \cdot(x'_{\widetilde P}-x'_P)
 \right|
 \leq
 Cr^{1+\beta}.
\end{equation}
Since
\[
 \nu(P)\cdot(x_{\widetilde P}-x_P)
 =
 \nu'(P)\cdot(x'_{\widetilde P}-x'_P)
 +
 \nu_m(P)\bigl(x_m(\widetilde P)-x_m(P)\bigr),
\]
multiplication of \eqref{eq:graph-first-order-remainder} by
\(\nu_m(P)\) gives
\[
 |\nu(P)\cdot(x_{\widetilde P}-x_P)|
 \leq
 Cr^{1+\beta}.
\]
This proves \eqref{eq:geometric-normal-flatness}.

In intrinsic coordinates centered at \(P\), the equivalence of \(\dK\)
with the centered graph gauge shows that
\eqref{eq:graph-first-order-remainder} gives
\eqref{eq:C1beta-remainder}, with
\[
 a_P
 =
 -{\nu'(P)}/{\nu_m(P)}.
\]
Furthermore, \eqref{eq:geometric-normal-holder} and
\eqref{eq:geometric-normal-transversality} imply
\[
 |a_P-a_{\widetilde P}|
 \leq
 C(\dK(P,\widetilde P))^{\beta_0}
 \leq
 C(\dK(P,\widetilde P))^\beta
\]
on the fixed bounded patch.  This is
\eqref{eq:C1beta-gradient}.  Hence \(S\) is locally
\(C_{\K}^{1,\beta}\), and the proof is complete.
\end{proof}

\begin{proof}[Proof of Theorem~\ref{thm:main}]
Theorem~\ref{thm:bowman-input} supplies a regular graph neighborhood,
unique half-space blow-ups, cone monotonicity, the required solution
estimates, and qualitative differentiability of the graph in the
Kolmogorov sense.  Estimate \eqref{eq:cylindrical-defect-decay} shows that
the graph is asymptotically cylindrical, so
Theorem~\ref{thm:BHI-input} applies uniformly below a fixed scale.

Proposition~\ref{prop:normal-holder} gives
\eqref{eq:normal-holder-main} with exponent \(\beta_0\).
Estimate \eqref{eq:y-flatness} verifies
\eqref{eq:geometric-y-hypothesis} with
\(\beta_y=1/2\).  Lemma~\ref{lem:flow-flatness}, after a harmless fixed
enlargement of the scale when \(|s|\leq Cr^2\), verifies
\eqref{eq:geometric-flow-hypothesis} with
\(\beta_t=\gamma/2\).  Applying
Lemma~\ref{lem:geometric-upgrade} with
\begin{equation}
\label{eq:final-beta}
 \beta
 =
 \min\bigl\{
 \beta_0,1/2,{\gamma}/{2}
 \bigr\}
 >0
\end{equation}
gives the quantitative flatness estimate
\eqref{eq:flatness-main} and shows that the regular free boundary is locally
\(C_{\K}^{1,\beta}\).

The identification of \(\nu(P)\) with the unique inward blow-up direction
was established in Proposition~\ref{prop:normal-holder}.  Finally, after
reducing \(r_*\) once more, \eqref{eq:normal-holder-main} ensures that $\nu(P)\cdot\nu(P_0)\geq 1/2$ throughout the stated neighborhood.  This completes the proof.
\end{proof}

\section{Further remarks}
\label{sec:remarks}

The proof separates three effects that become difficult to distinguish
if one formally differentiates the obstacle equation across the free
boundary.  First, the
geometry already available at regular points is sufficient for homogeneous
boundary comparison.  Indeed, the \(1/2\)-H\"older dependence of the graph
on \(y_m\) produces, after intrinsic rescaling, a critical
\(C^{0,1/3}_{y_m}\) cylindrical deviation of order \(r^{1/2}\).  Second, the
diffusion derivatives do not satisfy the homogeneous equation because of
the commutator terms in \eqref{eq:derivative-equations}.  After blow-up and
normalization, however, their sources are \(O(r)\), and this error is
summable in the quotient-oscillation iteration.  Third, the transport
derivative satisfies an equation with a divergence-form source whose
rescaled size is also \(O(r)\).  The resulting boundary decay of \(Yu\)
provides the power improvement in the intrinsic time direction.

It is natural to ask whether the commutator difficulty for the diffusion
derivatives could instead be avoided by using right-invariant vector
fields.  In coordinates centered at the group identity, these fields are
\begin{equation}
\label{eq:right-invariant-fields}
 \widehat X_i
 =
 \partial_{x_i}-t\partial_{y_i},
 \qquad i=1,\ldots,m,
\end{equation}
and a direct computation gives \([\K,\widehat X_i]=0\).  Since
\(\K u=1\) in \(\Omega_u\), it follows that
\(\K(\widehat X_i u)=0\) in \(\Omega_u\).  This commutation property
provides a useful formal parallel with the argument for step-two Carnot
groups in \cite{DanielliGarofaloPetrosyan2007}, but it does not supply the
boundary data required for the boundary Harnack argument used here.

To see the obstruction, fix
\(P_0=(x_0,y_0,t_0)\in\Gamma_u\) and use coordinates centered
intrinsically at \(P_0\).  The corresponding right-invariant field is
\begin{equation}
\label{eq:centered-right-invariant-field}
 \widehat X_i^{P_0}
 =
 \partial_{x_i}-(t-t_0)\partial_{y_i}.
\end{equation}
At \(P_0\), the coefficient \(t-t_0\) vanishes, and
Lemma~\ref{lem:gradient-vanishing} gives
\(\widehat X_i^{P_0}u(P_0)=u_{x_i}(P_0)=0\).
This vanishing is tied to the center \(P_0\).  If
\(\widetilde P\) is a neighboring free-boundary point and
\(Z\in\Omega_u\) approaches \(\widetilde P\), then
\begin{equation}
\label{eq:right-invariant-near-boundary}
 \widehat X_i^{P_0}u(Z)
 =
 u_{x_i}(Z)-(t_Z-t_0)u_{y_i}(Z).
\end{equation}
The first term tends to zero, but the available estimate
\(\nabla_yu\in\mathrm L^\infty_{\mathrm{loc}}\) does not provide a
continuous boundary trace for \(u_{y_i}\).  At most,
\eqref{eq:right-invariant-near-boundary} and the boundedness of
\(\nabla_yu\) give
\[
 \limsup_{\substack{Z\to\widetilde P,\,Z\in\Omega_u}}
 |\widehat X_i^{P_0}u(Z)|
 \leq
 C|t_{\widetilde P}-t_0|.
\]
This quantity is small when \(\widetilde P\) is close to \(P_0\), but it
need not vanish.  Homogeneous boundary Harnack requires exact continuous
vanishing on the relevant graph portion, not merely a small boundary
error.

One could recenter the right-invariant field at each boundary point
\(\widetilde P\), so that the coefficient of \(\partial_{y_i}\) vanishes
at \(\widetilde P\).  This, however, produces a different
\(\K\)-harmonic function at every boundary point and therefore does not
give a fixed family of positive \(\K\)-harmonic functions on a common
graph neighborhood.  In addition, cone monotonicity does not give
\(\widehat X_i^{P_0}u\) a fixed sign.  Thus the two properties needed for
the present boundary Harnack argument, positivity and continuous
vanishing on an entire graph portion, are unavailable for the
right-invariant derivatives.  If a continuous zero trace for
\(\nabla_yu\) were known, the boundary-vanishing obstruction would
disappear, but such a conclusion is not part of the regularity theory
used here.

By contrast, the left-invariant diffusion derivatives \(u_{x_i}\)
determine the non-degenerate normal, vanish continuously on the entire
regular graph, and have positive combinations \(q_j^\pm\) supplied by
cone monotonicity.  Their failure to be \(\K\)-harmonic is precisely the
scale-decaying deviation controlled by the
\(\K\)-harmonic-replacement argument.

The \(C_{\K}^{1,\beta}\) regularity established here raises the question
of further regularity of the regular graph.  In the constant-coefficient
setting considered in this paper, local smoothness is a natural
conjecture.  Such a conclusion would require a higher-order boundary
regularity argument beyond the quotient-oscillation estimates proved
here.  To indicate the issue, put
\[
 p=\partial_{x_m}u,
 \qquad
 R_j={\partial_{x_j}u}/{p},
 \qquad j=1,\ldots,m-1,
\]
in the portion of the positivity set under consideration, where \(p>0\).
The boundary traces satisfy
\[
 R_j\big|_{\Gamma_u}=-\partial_{x_j}\psi.
\]
Higher regularity of these quotients would therefore improve the
spatial regularity of the graph.  Their equations, however, retain the
commutator sources.  Indeed,
\[
 p\,\K R_j
 +2\nabla_xp\cdot\nabla_xR_j
 =
 -\partial_{y_j}u+R_j\partial_{y_m}u
 \qquad\text{in }\Omega_u.
\]
Since \(p\) vanishes at the free boundary, estimates for this equation
require a boundary theory adapted to its degeneracy.  The boundedness
of \(\nabla_yu\), which suffices for the \(O(r)\) perturbation argument,
does not by itself provide the higher-order estimates needed to
differentiate the quotient traces.

A possible approach would be to establish higher-order boundary
comparison estimates accommodating these sources, together with
improved boundary control of the transport derivatives.  The identities
\[
 \K(\partial_{y_i}u)=0,
 \qquad
 \K(\partial_tu)=0
 \qquad\text{in }\Omega_u
\]
suggest additional equations that could be used, but the required
boundary behavior of these derivatives would also have to be proved.
Moreover, improving the spatial normal alone would not establish full
higher intrinsic regularity.  The coupled time-transport direction
must be controlled as well.

A first concrete target is \(C_{\K}^{2,\gamma}\) regularity for some
\(\gamma\in(0,1)\).  In intrinsic coordinates centered at \(P\), this
would require a second-order approximation of the form
\[
 \left|
 \psi_P(x',y,t)
 -a_P\cdot x'
 -\frac12x'\cdot B_Px'
 -b_Pt
 \right|
 \leq
 C\bigl(|x'|+|y|^{1/3}+|t|^{1/2}\bigr)^{2+\gamma},
\]
uniformly on a smaller regular patch, with the corresponding
H\"older control and compatibility of the coefficients.  The time
variable enters this polynomial at homogeneous degree two, whereas
linear terms in \(y\) first enter at degree three.  Closing and
iterating such estimates would provide a route toward smoothness,
but the necessary higher-order boundary theory is not developed here.

The proof also suggests possible extensions to lower-order perturbations
and divergence-form Kolmogorov operators.  Such extensions would require
boundary comparison estimates that are uniform in the relevant
asymptotically cylindrical class, together with sufficient decay of the
rescaled derivative sources.  For variable leading coefficients, one
would additionally need a modulus of continuity ensuring that the errors
created by differentiating or freezing the coefficients are summable in
the analogue of \eqref{eq:inhomogeneous-contraction}.  These questions are
not pursued here.

Our main theorem concerns only the regular part of the free boundary.  The
structure and stratification of the singular set remain separate problems.
The reduction in \cite{Bowman2025} of the blow-up equation to the classical
parabolic obstacle problem suggests that methods from parabolic
singular-set theory should remain relevant.  Nevertheless, the transport
variables and the non-Euclidean group law introduce additional
compatibility conditions that do not arise in the analysis of regular
points.

\medskip
\noindent
\textbf{Declaration on the use of generative AI.}
The author used ChatGPT, developed by OpenAI, as an interactive tool
during the preparation and revision of this manuscript.  The tool was
used for language editing, organization, LaTeX formatting, and checks of
clarity and internal consistency.  The mathematical arguments, results,
and conclusions were developed and verified by the author, who also
checked the references and takes full responsibility for the content of
the manuscript.

\end{document}